\documentclass{amsart}

\usepackage{amsmath,amssymb}
\usepackage{amsthm}
\usepackage[nobysame]{amsrefs}
\usepackage{indentfirst}
\usepackage{mathrsfs}
\usepackage{graphicx}
\usepackage{hyperref}
\usepackage{xcolor}
\usepackage{fourier}
\usepackage[margin=1.5in]{geometry}

\theoremstyle{plain}
\newtheorem{theorem}{Theorem}
\newtheorem{lemma}{Lemma}

\newtheorem{prop}[lemma]{Proposition}

\theoremstyle{definition}
\newtheorem{defn}[lemma]{Definition}

\theoremstyle{remark}
\newtheorem*{remark}{Remark}

\DeclareMathOperator{\dist}{dist}
\DeclareMathOperator{\argmin}{arg min}

\DeclareMathOperator{\divop}{div}

\newcommand{\ie}{\textit{i.e.}}

\newcommand{\ud}{\,\mathrm{d}}
\newcommand{\rd}{\mathrm{d}}

\newcommand{\RR}{\mathbb{R}}

\newcommand{\Or}{\mathcal{O}}

\newcommand{\wt}[1]{\widetilde{#1}}
\newcommand{\wh}[1]{\widehat{#1}}

\IfFileExists{mathabx.sty}%
  {\DeclareFontFamily{U}{mathx}{\hyphenchar\font45}%
   \DeclareFontShape{U}{mathx}{m}{n}{<->mathx10}{}%
   \DeclareSymbolFont{mathx}{U}{mathx}{m}{n}%
   \DeclareFontSubstitution{U}{mathx}{m}{n}%
   \DeclareMathAccent{\widebar}{0}{mathx}{"73}%
}{%
  \PackageWarning{mathabx}{%
    Package mathabx not available, therefore\MessageBreak substituting
    widebar with overline\MessageBreak }%
  \newcommand{\widebar}[1]{\overline{#1}}%
}
\newcommand{\wb}[1]{\widebar{#1}}

\newcommand{\mc}[1]{\mathcal{#1}}
\newcommand{\ms}[1]{\mathscr{#1}}
\newcommand{\mf}[1]{\mathsf{#1}}

\newcommand{\abs}[1]{\lvert#1\rvert}

\newcommand{\norm}[1]{\left\lVert#1\right\rVert}

\newcommand{\mres}{\mathbin{\vrule height 1.3ex depth 0pt width
0.13ex\vrule height 0.13ex depth 0pt width 1ex}}

\newcommand{\es}{\mathrm{es}}
\newcommand{\eff}{\mathrm{eff}}

\title{An isoperimetric problem with Coulomb repulsion and
  attraction to a background nucleus}

\author{Jianfeng Lu} \address{Departments of Mathematics, Physics, and
  Chemistry, Duke University, Durham, NC 27708, USA}
\email{jianfeng@math.duke.edu}

\author{Felix Otto} \address{Max Planck Institute for Mathematics in
  the Sciences, Inselstr. 22 04103 Leipzig, Germany} \email{otto@mis.mpg.de}

\date{\today}

\thanks{The research of J.L.~was supported in part by the Alfred
  P.~Sloan Foundation and the National Science Foundation grant
  DMS-1312659. J.L.~would like to thank the warm hospitality of the
  Max Planck Institute for Mathematics in the Sciences, where part of
  the work is done.}
\begin{document}

\begin{abstract}
  We study an isoperimetric problem the energy of which contains the
  perimeter of a set, Coulomb repulsion of the set with itself, and
  attraction of the set to a background nucleus as a point charge with
  charge $Z$. For the variational problem with constrained volume $V$,
  our main result is that the minimizer does not exist if $V - Z$ is
  larger than a constant multiple of $\max\bigl(Z^{2/3}, 1\bigr)$. The
  main technical ingredients of our proof are a uniform density lemma
  and electrostatic screening arguments.
\end{abstract} 

\maketitle

In this work, we study an energy functional for three-dimensional sets $\Omega \subset \RR^3$, given by 
\begin{equation}\label{eq:NLIP}
  \mf{E}_Z(\Omega) 
  := \abs{\partial \Omega} + \frac{1}{2} \iint_{\Omega \times \Omega} 
  \frac{1}{\abs{x - y}} \ud x \ud y - \int_{\Omega} \frac{Z}{\abs{x}},
\end{equation}
where $\abs{\partial\Omega}$ is the perimeter of the set $\Omega$ and
$Z \geq 0$ is given.  Our main result is the following theorem.
\begin{theorem}\label{thm:nonexist}
  There exists a universal constant $M$ such that the variational problem
  \begin{equation}\label{eq:NLIPvar}
    E_Z(V) = \inf_{\abs{\Omega} = V} \mf{E}_Z(\Omega)
  \end{equation}
  does not have a minimizer if $V \geq Z + M \max(Z^{2/3}, 1)$. 
\end{theorem}
On the other hand, when the volume of the set is small, the minimizer
of the variational problem exists and is given by a ball centered at
the origin.
\begin{theorem}\label{thm:ball}
  There exists a universal constant $m$ such that for any $V \leq Z + m$, the
  unique minimizer of the variational problem
  \begin{equation*}
    E_Z(V) = \inf_{\abs{\Omega} = V} \mf{E}_Z(\Omega)
  \end{equation*}
  for $Z > 0$ is given by $\Omega = B_{(V / \abs{B_1})^{1/3}}(0)$.
\end{theorem}

The study of the variational problem \eqref{eq:NLIPvar} is a natural
extension of our previous work \cite{LuOtto:14} (see also the study by
Kn\"upfer and Muratov in \cites{KnupferMuratov, KnupferMuratov:14} and
a recent work by Frank and Lieb in \cite{FrankLieb}), in particular,
Theorem~\ref{thm:nonexist} is an extension of the nonexistence result
when $Z = 0$ \cite{LuOtto:14}*{Theorem 2}. It turns out that our proof
of the extension, presented in Section~\ref{sec:nonexist}, requires
several new ideas. One of the essential ideas of our current proof is
to explore the fact that if $\Omega$ has a large volume, far away from
the nucleus, the nucleus charge should be (electrically) screened, and
hence we are back to the situation without a background potential. To
make this idea work, we need to analyze the electrostatic part of the
problem and to use the following density lemma.
\begin{lemma}[Density Lemma]\label{lem:density}
  Let $\Omega$ be a minimizer of \eqref{eq:NLIPvar} of prescribed
  volume $V$. Then for all $x$ such that $x \in \Omega$ in the
  measure-theoretic sense, namely
  \begin{equation}\label{eq:belongset}
    \abs{\Omega \cap B_r(x)} > 0, \quad \forall\, r > 0,
  \end{equation}
  and such that it holds
  \begin{equation*}
    \phi_x(y) := \int_{\Omega - B_1(x)} \frac{1}{\abs{y - z}} \ud z - \frac{Z}{\abs{y}} > 0,
    \qquad \forall \, y \in B_1(x),
  \end{equation*}
  we have
  \begin{equation}\label{eq:density}
    \abs{\Omega \cap B_1(x)} \geq \delta, 
  \end{equation}
  where $\delta > 0$ is a universal constant.
\end{lemma}
Roughly speaking, the lemma states that the minimizing set cannot be
too ``thin'' where the Coulomb potential generated by the nuclear
charge is screened.  We emphasize that the constant $\delta$ in the
above density lemma is a universal constant, in particular, it does
not depend on the gradient of the Coulomb potential generated by
$\Omega$ which is potentially large in the ball.

The proof of Lemma~\ref{lem:density} is similar in spirit to that of
the density lemma~\cite{LuOtto:14}*{Lemma 4} in our previous
work. However, the previous argument does not apply as competitors are
constructed by deforming the set by a global dilation, which might
tremendously increase the energy in the current case and hence is not
useful. We thus have to turn to more delicate arguments that utilize
local deformations of the set. The details are given in
Section~\ref{sec:density}.

The version of Theorem~\ref{thm:ball} when $Z = 0$ was proved in
\cite{KnupferMuratov:14} and later extended to various settings by
\cites{Julin:14, MuratovZaleski, FigalliFusco, BonaciniCristoferi}
(see also related works in \cites{CicaleseSpadaro:13, Topaloglu:13,
  SternbergTopaloglu:11, ChoksiSternberg:07, AcerbiFuscoMorini:13,
  GoldmanNovagaRuffini}).  Theorem~\ref{thm:ball} extends this type of
results in the presence of an external potential. Our proof is close
to the idea in \cite{Julin:14} which uses a version of a quantitative
isoperimetric inequality that measures the deficit of a set from a
ball by the Coulomb potential. The proof is given in
Section~\ref{sec:ball}.

\medskip 

From the point of view of physics, Theorem~\ref{thm:nonexist} is related to the ionization
conjecture in quantum mechanics, which states that the number of
electrons that can be bound to an atomic nucleus of charge $Z$ cannot
exceed $Z + 1$. Theorem~\ref{thm:nonexist} gives an upper bound of the
volume of the set for the nucleus of charge $Z$ in the nonlocal
isoperimetric model. The ionization conjecture, while still open for
the Schr\"odinger equation, has been studied by many authors for
different types of models in quantum mechanics (see
e.g. \cites{BenguriaBrezisLieb:81, BenguriaLieb:85, FeffermanSeco:90,
  Lieb:84, LiebSigalSimonThirring:88, Sigal:82, Solovej:91,
  Solovej:03}). Our study is motivated by this question, as the model
\eqref{eq:NLIP} can be understood as a ``sharp interface version'' of
the Thomas-Fermi-Dirac-von~Weisz\"acker (TFDW) model, a mean field type
approximation of the many-body Schr\"odinger equation. See
\cite{LuOtto:14} for more remarks on the connection of the nonlocal
isoperimetric model to the TFDW model.

\smallskip 

\textbf{Notation.} For $r > 0$, we denote by $B_r(x)$ the ball with
center $x$ and radius $r$, and $B_r$ if the center is the origin. For
$0 < r_1 < r_2$, we denote $A_{r_1, r_2}(x)$ the annulus with center
$x$, inner radius $r_1$ and outer radius $r_2$. Similarly $A_{r_1,
  r_2}$ if the center is the origin. $C$ and $c$ denote generic
constants (in particular, independent of $Z$ and $V$) whose value may
change from line to line. Moreover, to avoid specifying unnecessary
constants, we will use the notation $x \lesssim y$ which stands for $x
\leq C y$ for some universal constant $C$. We write $x \sim y$ if both
$x\lesssim y$ and $x \gtrsim y$ hold. Finally, we use the notation $x
\ll y$ to denote that $y \geq C x$ for a sufficiently large constant
$C$.

\section{Proof of Theorem~\ref{thm:nonexist}}
\label{sec:nonexist}

In this section we prove Theorem~\ref{thm:nonexist} assuming the
Density Lemma~\ref{lem:density}. The proof of the latter is given in
Section~\ref{sec:density}. We will assume $Z > 0$, as the case $Z = 0$
is treated in our previous work \cite{LuOtto:14}. In fact, to ease the
presentation, we will assume that $Z$ is sufficiently large, which is
the most interesting scenario for Theorem~\ref{thm:nonexist}. We will
also make the standing assumption that $V - Z \gg Z^{2/3}$ since
otherwise there is nothing to prove.

\subsection{Electrostatic energy}
The first step is to consider the electrostatic part of the energy
\begin{equation}
  \mf{E}_{\es}(\Omega) := 
  \frac{1}{2} \iint_{\Omega\times\Omega} \frac{1}{\abs{x-y}} \ud x \ud y -
  \int_{\Omega} \frac{Z}{\abs{x}} \ud x. 
\end{equation}
We define a radius $R_Z>0$ such that $\abs{B_{R_Z}} = Z$,
\textit{i.e.,}~$R_Z = (Z/\abs{B_1})^{1/3}$.  Let $E_{\es}(Z)$ be the
electrostatic energy of $B_{R_Z}$:
\begin{equation}
  E_{\es}(Z) = \frac{1}{2} \iint_{B_{R_Z}\times B_{R_Z}} \frac{1}{\abs{x-y}} 
  \ud x \ud y - \int_{B_{R_Z}} \frac{Z}{\abs{x}} \ud x. 
\end{equation}
To simplify notation, in the sequel, we denote $\chi_Z :=
\chi_{B_{R_Z}}$ the characteristic function of the ball $B_{R_Z}$.

For the electrostatic problem, it is natural to introduce 
\begin{defn} The Coulomb norm of $f$ is given by 
  \begin{equation*}
    \norm{f}_C := \Biggl( \iint \frac{f(x)f(y)}{\abs{x - y}} \ud x \ud y \Biggr)^{1/2} = \Biggl( 4\pi \int \frac{\abs{\wh{f}(k)}^2}{\abs{k}^2} \ud k  \Biggr)^{1/2}.
  \end{equation*}
\end{defn}
The non-negativeness of the norm is obvious from the Fourier
representation. The Fourier representation also yields the duality
with the homogeneous Sobolev space $\dot{H}^1(\RR^3)$:
\begin{equation}\label{eq:coulombduality}
  \int f \psi \leq \frac{1}{\sqrt{4\pi}} \norm{f}_C \Biggl( \int \abs{\nabla \psi}^2 \Biggr)^{1/2}.
\end{equation}

The next proposition states that the ball $B_{R_Z}$ minimizes the
electrostatic part of the energy.
\begin{prop}\label{prop:Ees}
  Provided $V \geq Z$, we have 
  \begin{equation}\label{eq:Ees}
    \inf_{\abs{\Omega} = V} \mf{E}_{\es}(\Omega) = \inf_{\abs{\Omega} = Z} 
    \mf{E}_{\es}(\Omega) = E_{\es}(Z).
  \end{equation}
\end{prop}

\begin{proof}
  For any $\Omega \subset \RR^3$ with finite volume, 
  \begin{equation}\label{eq:escalc}
    \begin{aligned}
      \mf{E}_{\es}(\Omega) - \mf{E}_{\es}(B_{R_Z}) & = \frac{1}{2}
      \iint \frac{\chi_{\Omega}(x) \chi_{\Omega}(y) - \chi_Z(x)
        \chi_Z(y)}{\abs{x-y}} \ud x \ud y -
      Z \int \frac{\chi_{\Omega}(x) - \chi_Z(x)}{\abs{x}} \ud x \\
      & = \frac{1}{2} \iint \frac{\bigl(\chi_{\Omega}(x) - \chi_Z(x)\bigr)\bigl(\chi_{\Omega}(y) - \chi_Z(y)\bigr)}{\abs{x - y}} \ud x \ud y \\
      & \qquad + \iint \bigl(\chi_{\Omega}(x) - \chi_Z(x)\bigr)
      \frac{\chi_Z(y)}{\abs{x - y}}
      \ud x \ud y - Z \int \frac{\chi_{\Omega}(x) - \chi_Z(x)}{\abs{x}} \ud x \\
      & = \frac{1}{2} \norm{\chi_{\Omega} - \chi_Z}_C^2 + \int \bigl(
      \chi_{\Omega}(x) - \chi_Z(x) \bigr) u(x) \ud x \\
      & \geq \int \bigl( \chi_{\Omega}(x) - \chi_Z(x) \bigr) u(x) \ud
      x,
    \end{aligned}
  \end{equation}
  where we have introduced $u(x)$
  defined as 
  \begin{equation*}
    u(x) = \int \frac{\chi_Z(y)}{\abs{x - y}} \ud y - \frac{Z}{\abs{x}}
    =  \begin{cases}
      2 \pi R_Z^2 + (\abs{B_1} - 2 \pi) \abs{x}^2 - \frac{Z}{\abs{x}}, & \abs{x} \leq R_Z; \\
      0, & \abs{x} \geq R_Z.
    \end{cases}
  \end{equation*}
  Therefore, we have 
  \begin{equation*}
    \mf{E}_{\es}(\Omega) - \mf{E}_{\es}(B_{R_Z}) \geq \int_{B_{R_Z}} \bigl( \chi_{\Omega}(x) - \chi_Z(x) \bigr) u(x) \ud x \geq 0 
  \end{equation*}
  as both $\chi_{\Omega} - \chi_Z$ and $u$ are non-positive
  inside $B_{R_Z}$. This implies for any $V\geq 0$, 
  \begin{equation*}
    \inf_{\abs{\Omega} = V} \mf{E}_{\es}(\Omega) \geq E_{\es}(Z),
  \end{equation*}
  and hence 
  \begin{equation*}
    \inf_{\abs{\Omega} = Z} \mf{E}_{\es}(\Omega) = E_{\es}(Z).
  \end{equation*}
  To get the first equality in \eqref{eq:Ees}, notice that for $V
  \geq Z$, we have 
  \begin{equation*}
    \inf_{\abs{\Omega} = V} \mf{E}_{\es}(\Omega) \leq \inf_{\abs{\Omega} = Z} 
    \mf{E}_{\es}(\Omega)
  \end{equation*}
  as we can break the excess volume ($V-Z$) into tiny balls and
  place them far away from each other and from the origin, such that
  the Coulomb interaction between them is made arbitrarily small.
\end{proof}

Going back to the full energy, we note an easy upper bound for the
minimum energy of the variational problem \eqref{eq:NLIPvar}.
%
\begin{lemma}\label{lem:Eupperbound}
  For $E_Z(V)$ given as the infimum of the variational problem
  \eqref{eq:NLIPvar}, we have for $0 \leq V' \leq V$ 
  \begin{equation}\label{eq:linearbound}
    E_Z(V) \leq E_Z(V') + E_0(V- V'), 
  \end{equation}
  where $E_0$ denotes $E_Z$ for $Z = 0$. Furthermore, for $E_0$, we
  have
  \begin{equation}\label{eq:E0bound}
    E_0(V) \lesssim  V^{2/3} + V. 
  \end{equation}
\end{lemma}

\begin{proof}
  The proof of \eqref{eq:linearbound} follows a similar argument of
  \cite{LuOtto:14}*{Lemma 3(i)} in our previous work. Assuming
  $\Omega_1$ and $\Omega_2$ are bounded sets that approximately
  minimize $E_Z(V')$ and $E_0(V - V')$ respectively, the inequality is
  obtained by considering a test set for $E_Z(V)$ that consists of
  $\Omega_1 \cup (\Omega_2 + d)$ with a large shift vector $d$. We
  refer the readers to the proof of \cite{LuOtto:14}*{Lemma 3(i)} for
  details.

  For $V \ll 1$, the estimate \eqref{eq:E0bound} follows by taking a
  ball with volume $V$ as the test set. For $V \gtrsim 1$, the
  estimate follows from the sub-additivity by taking $Z = 0$ in
  \eqref{eq:linearbound}.
\end{proof}

Using Lemma~\ref{lem:Eupperbound}, we now prove
\begin{lemma}\label{lem:EupperboundZ}
  We have 
  \begin{equation}
    E_Z(V) - E_{\es}(Z) \lesssim V - Z.
  \end{equation}
\end{lemma}
\begin{proof} 
  
  By taking $V' = Z$ in \eqref{eq:linearbound}, we get 
  \begin{equation*}
    E_Z(V) - E_Z(Z) \leq E_0(V - Z) \lesssim V - Z, 
  \end{equation*}
  where the last inequality follows from \eqref{eq:E0bound} and the
  assumption that $V - Z \gg Z^{2/3} \gtrsim 1$. Observing that 
  \begin{equation*}
    E_Z(Z) \leq \mf{E}_Z(B_{R_Z}) = E_{\es}(Z) + \abs{ \partial B_{R_Z} }, 
  \end{equation*}
  we arrive at the conclusion since $\abs{\partial B_{R_Z}} \sim
  Z^{2/3} \ll V - Z$.
\end{proof}

With the above a priori bound for the energy, we are now
ready to state and prove the main estimate of this subsection. The
estimate states that the minimizer is ``close'' to the ball $B_{R_Z}$
near the origin. More precisely, we have 
\begin{prop}\label{prop:closeball}
  Let $\Omega$ be a minimizer of \eqref{eq:NLIPvar}. Then
  \begin{equation}\label{eq:BROmega}
    \abs{B_{R_Z} - \Omega} \lesssim 
     (V - Z)^{1/2} Z^{1/3}, 
  \end{equation}
  and also
  \begin{equation}\label{eq:B2ROmega}
    \abs{\Omega \cap B_{2R_Z}} - Z \lesssim (V - Z)^{1/2} Z^{1/3}. 
  \end{equation}
\end{prop}

Note that by definition $\abs{B_{R_Z}} = Z$, so if $V$ is not too
large, \eqref{eq:BROmega} states that the minimizing set $\Omega$
almost fills the ball $B_{R_Z}$ (the difference is of lower order) and
\eqref{eq:B2ROmega} states that the minimizing set $\Omega$ has a
small volume in the annulus $A_{R_Z, 2R_Z}$. 

\begin{proof}
  Without loss of generality, we assume that $Z > \abs{B_1}$ such that
  $R_Z > 1$.  Let $\psi$ be the function
  \begin{equation*}
    \psi(x) = 
    \begin{cases}
      1, & \abs{x} \leq R_Z - 1; \\
      R_z - \abs{x}, & R_Z - 1 \leq \abs{x} \leq R_Z; \\
      0, & \abs{x} \geq R_Z.
    \end{cases}
  \end{equation*}
  We estimate
  \begin{equation}\label{eq:missvol}
    \begin{aligned}
      \abs{B_{R_Z} - \Omega} & = \int_{B_{R_Z}}
      (\chi_Z - \chi_{\Omega}) \\
      & = \int_{B_{R_Z}} (\chi_Z - \chi_{\Omega}) \psi +
      \int_{B_{R_Z}} (\chi_Z - \chi_{\Omega}) (1 - \psi) \\
      & \stackrel{\eqref{eq:coulombduality}}{\leq} \frac{1}{\sqrt{4\pi}}
      \norm{\chi_Z - \chi_{\Omega}}_C \Bigl( \int\abs{\nabla \psi}^2
      \Bigr)^{1/2} + \int_{B_{R_Z} - B_{R_Z - 1}} 1 \\
      & \lesssim \norm{\chi_Z - \chi_{\Omega}}_C R_Z + R_Z^2.
    \end{aligned}
  \end{equation}
  To control $\norm{\chi_Z - \chi_{\Omega}}_C$, recall the calculation
  \eqref{eq:escalc}
  \begin{equation*}
    \begin{aligned}
      \mf{E}_{\es}(\Omega) - E_{\es}(Z)  = \mf{E}_{\es}(\Omega) -
      \mf{E}_{\es}(B_{R_Z}) 
       = \frac{1}{2} \norm{ \chi_{\Omega} - \chi_Z}_C^2  + \int
      (\chi_{\Omega} - \chi_Z) u.
    \end{aligned}
  \end{equation*}
  As shown in the proof of Proposition~\ref{prop:Ees}, the second term
  on the right hand side is non-negative, and hence
  \begin{equation}\label{eq:Coulombbound}
    \norm{\chi_{\Omega} - \chi_Z}_C^2 \leq 2 (\mf{E}_{\es}(\Omega) - 
    E_{\es}(Z)) \leq 2(\mf{E}_Z(\Omega) - E_{\es}(Z)) \lesssim V - Z,
  \end{equation}
  where we have used Lemma~\ref{lem:EupperboundZ} in the last
  inequality.

  Using the bound \eqref{eq:Coulombbound} in \eqref{eq:missvol}, we
  arrive at (recall that $V-Z \gg Z^{2/3}$)
  \begin{equation*}
    \abs{B_{R_Z} - \Omega} \lesssim
    (V - Z)^{1/2} Z^{1/3} + Z^{2/3} \lesssim (V - Z)^{1/2} Z^{1/3}.
  \end{equation*}

  The proof of \eqref{eq:B2ROmega} is similar. We replace $\psi$ by 
  \begin{equation*}
    \varphi(x) = 
    \begin{cases}
      1, & \abs{x} \leq 2 R_Z-1; \\
      2 R_Z - \abs{x}, & 2 R_Z-1 \leq \abs{x} \leq 2 R_Z; \\
      0, & \abs{x} \geq 2 R_Z.
    \end{cases}
  \end{equation*}
  Then, 
  \begin{equation*}
    \begin{aligned}
      \abs{\Omega \cap B_{2R_Z}} & \leq \int \chi_{\Omega}
      \varphi \\
      & = \int (\chi_{\Omega} - \chi_Z) \varphi + \int \chi_Z \varphi \\
      & \stackrel{\eqref{eq:coulombduality}}{\leq} \frac{1}{\sqrt{4\pi}} \norm{\chi_Z - \chi_{\Omega}}_C \Bigl( \int\abs{\nabla
        \varphi}^2 \Bigr)^{1/2} + Z \\
      & \leq C \norm{\chi_Z - \chi_{\Omega}}_C R_Z + Z.
    \end{aligned}
  \end{equation*}
  The estimate \eqref{eq:B2ROmega} follows. 
\end{proof}

\subsection{Charge screening}

The central idea of our proof is to screen the background nucleus
charge, so that we may use the argument for the case without the
nucleus from \cite{LuOtto:14}. We set the effective nucleus charge as
\begin{equation*}
  Z_{\eff} =  (V-Z)^{1/2} Z^{1/3}. 
\end{equation*}
Note that if $V - Z \sim Z^{2/3}$, we have $Z_{\eff} \sim Z^{2/3}$.  By
\eqref{eq:BROmega}, we have
\begin{equation}\label{eq:surpluscharge}
  Z - \abs{\Omega \cap B_{R_Z}} = \abs{B_{R_Z} - \Omega}
  \leq c_0 Z_{\eff},
\end{equation}
where $c_0$ is a universal constant. Note that by Newton's theorem,
outside the radius $R_Z$, the nucleus charge is screened by the amount
of positive charge in $B_{R_Z}$, which is  given by $\abs{\Omega \cap
  B_{R_Z}}$. 

We will establish the following estimate of the volume of a minimizer
$\Omega$, which justifies the terminology of effective nucleus charge.
\begin{prop}\label{prop:refined}
  Let $\Omega$ be a minimizer of \eqref{eq:NLIPvar} of volume $V$, then 
  \begin{equation}\label{eq:refined}
    V - Z \lesssim Z_{\eff}.
  \end{equation}
\end{prop}

Note that Theorem~\ref{thm:nonexist} is an easy corollary of
Proposition~\ref{prop:refined}: By \eqref{eq:refined} and the
definition of $Z_{\eff}$, we have
\begin{equation*}
  V - Z \lesssim (V-Z)^{1/2} Z^{1/3}, 
\end{equation*}
thus indeed $V-Z \lesssim Z^{2/3}$.  
It suffices to prove Proposition~\ref{prop:refined}. Note that we have
$V - Z \gg Z_{\eff}$ since it is assumed (for proof by contradiction)
that $V - Z \gg Z^{2/3}$.

\smallskip

We start the proof of Proposition~\ref{prop:refined} by showing that
the nucleus charge is completely screened far away from the
origin. However, Lemma~\ref{lem:screenC} below is qualitative in the
sense that it yields no control on the radius $R$ starting from which
we have complete screening. In view of Lemma~\ref{lem:density}, we
strengthen this complete screening in terms of the modified potential
$\phi_x$ that ignores the effect of the nearby charges: For
$y \in B_1(x)$
\begin{equation*}
  \phi_x(y) = \int_{\Omega - B_1(x)} \frac{1}{\abs{y - z}} \ud z - \frac{Z}{\abs{y}}. 
\end{equation*}
Hence $\phi_x$ is the electric potential generated by the nucleus and
the set $\Omega$ outside the ball $B_1(x)$.

\begin{lemma}\label{lem:screenC}
  There exists a radius $R \geq 2 R_Z$ such that
  \begin{equation*}
     \phi_x \vert_{B_1(x)}  > 0 \ \text{ if }\ \abs{x} \geq R, \quad \text{and} \quad 
     \abs{\Omega \cap B_R} - Z \lesssim Z_{\eff}.
  \end{equation*}
\end{lemma}


\begin{proof}
  We take a radius $\wt{R}$ such that
  \begin{equation*}
    \abs{\Omega \cap A_{R_Z, \wt{R}}} = 6 c_0 Z_{\eff}, 
  \end{equation*}
  where $c_0$ is the universal constant in
  \eqref{eq:surpluscharge}. Note that if for all $\wt{R} < \infty$,
  $\abs{\Omega \cap A_{R_Z, \wt{R}}} < 6 c_0 Z_{\eff}$, we would have
  \begin{equation*}
    V - Z \leq \abs{\Omega - B_{R_Z}} \leq 6 c_0 Z_{\eff}, 
  \end{equation*}
  which contradicts with the assumption $V - Z \gg Z_{\eff}$.  Set $R =
  \max( \wt{R}, 2 R_Z)$; this choice of $R$ guarantees that $R \geq
  2R_Z$. By \eqref{eq:B2ROmega} in Proposition~\ref{prop:closeball}
  and the choice of $\wt{R}$, we have
  \begin{equation}\label{eq:OmegaRZR}
    6 c_0 Z_{\eff} \leq \abs{\Omega \cap A_{R_Z, R}} \lesssim Z_{\eff}.
  \end{equation}
  Together with \eqref{eq:BROmega}, this implies $\abs{\Omega \cap
    B_{R}} - Z = \abs{B_{R_Z} - \Omega} + \abs{\Omega \cap A_{R_Z, R}} \lesssim Z_{\eff}$.  Thus, it remains to verify that
  $\phi_x \vert_{B_1(x)} > 0$ if $\abs{x} > R$.  Let us start with an
  elementary estimate for $y \not\in B_{R_Z}$
  \begin{multline}
    \label{eq:screenrz}
    \int_{\Omega \cap B_{R_Z}} \frac{1}{\abs{y-z}} \ud z -
    \frac{Z}{\abs{y}}  = \int_{\Omega \cap B_{R_Z}} \frac{1}{\abs{y-z}} \ud z - \int_{B_{R_Z}} \frac{1}{\abs{y - z}} \ud z \\
    = - \int_{B_{R_Z} - \Omega} \frac{1}{\abs{y-z}} \ud z
    \stackrel{\eqref{eq:surpluscharge}}{\geq} \frac{-c_0
      Z_{\eff}}{\abs{y}-R_Z}.
  \end{multline}
  Hence, we have the  estimate
  \begin{equation*}
    \begin{aligned}
      \phi_x(y) & = \int_{\Omega - B_1(x)} \frac{1}{\abs{y-z}} \ud z - \frac{Z}{\abs{y}} \\
      & \geq \biggl( \int_{\Omega \cap B_{R_Z}} \frac{1}{\abs{y-z}}
      \ud z - \frac{Z}{\abs{y}} \biggr) + \int_{\Omega \cap ( A_{R_Z,
          R} - B_1(x))} \frac{1}{\abs{y-z}} \ud z  \\
      & \stackrel{\eqref{eq:screenrz}}{\geq} -
      \frac{c_0Z_{\eff}}{\abs{y} - R_Z} + \frac{\abs{\Omega
          \cap A_{R_Z, R}} - \abs{\Omega \cap B_1(x)}}{\abs{y} + R } \\
      & \stackrel{\eqref{eq:OmegaRZR}}{\geq} -
      \frac{c_0Z_{\eff}}{\abs{y} - R_Z} + \frac{6 c_0 Z_{\eff} -
        \abs{B_1}} {\abs{y} + R}.
    \end{aligned}
  \end{equation*}
  Therefore, as $Z_{\eff} \gg \abs{B_1}$, we have $\phi_x
  \vert_{B_1(x)} > 0$ provided that
  \begin{equation*}
    5 (\abs{y} - R_Z) > \abs{y} + R 
  \end{equation*}
  for $\abs{y} \geq \abs{x} - \abs{y - x} \geq R - 1$, which can be
  easily checked.
\end{proof}

\bigskip

\begin{figure}[ht]
\includegraphics[width = 0.95\textwidth]{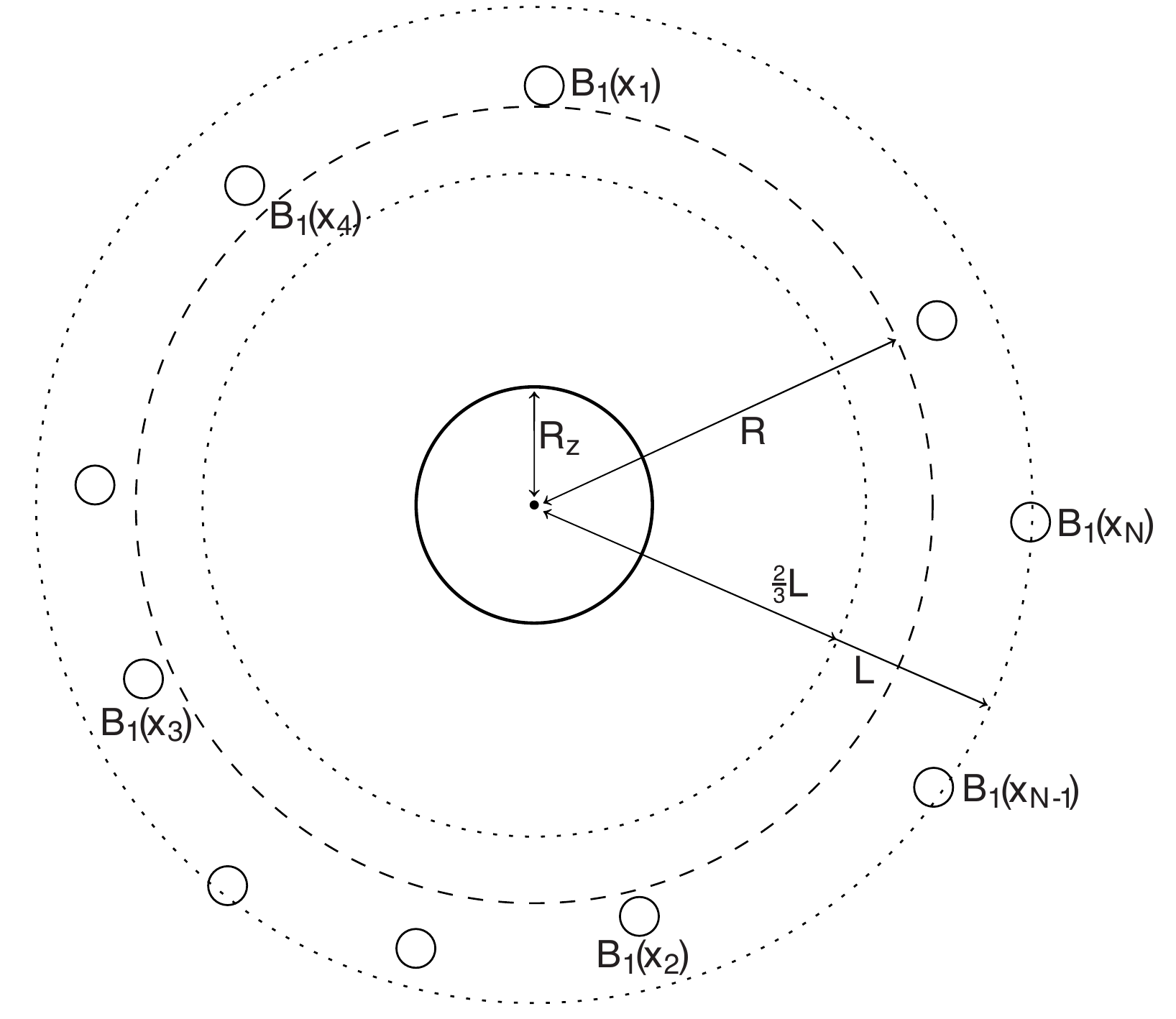}
\caption{Screening the nucleus charge by the set of balls
  $B_1(x_1), \ldots, B_1(x_N)$. \label{fig:screening}}
\end{figure}

Because of $\abs{\Omega \cap B_R} - Z \lesssim Z_{\eff}$ and
$Z_{\eff} \ll V-Z$ by assumption, there is still substantial charge
outside of radius $R$. Of this charge we consider the amount of order
$Z_{\eff}$ closest to the origin, roughly speaking.  As $\phi_x > 0$
in $B_1(x)$ for $\abs{x} \geq R$ by Lemma~\ref{lem:screenC}, we can
now apply the Density Lemma~\ref{lem:density} to see that this excess charge is not too scattered. We will use this to further screen the
nucleus charge.  Let $N$ be the smallest positive integer which
satisfies $N \delta \geq 13 c_0 Z_{\eff}$, where $\delta$ is the
universal constant in \eqref{eq:density}. We take a sequence of points
$\{x_1, \cdots, x_N\}$ in $\Omega$ recursively defined as follows (see
Figure~\ref{fig:screening}):
\begin{equation*}
  x_1 = \argmin\, \bigl\{ \abs{x} \;\big\vert\; \abs{x} \geq R, 
  \ x \in \Omega \text{ in the sense of \eqref{eq:belongset}}\, \bigr\}, 
\end{equation*}
and for $k \leq N$, after $x_j$ is chosen for $j \leq k - 1$, we take
\begin{equation*}
  x_k = \argmin\, \bigl\{ \abs{x} \;\big\vert\;  \abs{x} \geq R,
  \, \min_{j \leq k-1} \abs{x - x_j} \geq 2, 
\ x \in \Omega \text{ in the sense of \eqref{eq:belongset}} \,\bigr\}.
\end{equation*}
Since $\abs{\Omega - B_R} = V - \abs{\Omega \cap B_R} \gg Z_{\eff} \gg 1$
from Lemma~\ref{lem:screenC}, such a sequence $\{x_k\}$ can be selected.
Also by Lemma~\ref{lem:screenC}, for each $x_k$, $\phi_{x_k} > 0$ in
$B_1(x_k)$.
Using Lemma~\ref{lem:density}, we obtain
\begin{equation}\label{eq:chargeinballs}
  \Bigl \lvert \Omega \cap \bigcup_{k=1}^N  B_1(x_k) \Bigr\rvert
  = \sum_{k=1}^N \bigl\lvert B_1(x_k) \cap \Omega \bigr\rvert 
  \geq N \delta \geq 13 c_0 Z_{\eff}. 
\end{equation}
Defining
\begin{equation*}
  L = \max_{1 \leq k \leq N} \abs{x_k} \geq R,
\end{equation*}
we now further refine our screening estimate by using the balls
centered at $\{x_k\}$.
\begin{lemma}\label{lem:screenBall}
  We have for all $x$ with $\abs{x} \geq \frac{2}{3}L$
  \begin{equation}\label{eq:screenBall}
    \int_{\Omega \cap (B_{\frac{2}{3}L} \cup \;\bigcup_k B_1(x_k))} \frac{1}{\abs{x - y}} \ud y
    - \frac{Z}{\abs{x}} > 0.
  \end{equation}
\end{lemma}
The main point of Lemma~\ref{lem:screenBall} is that the additional
charges concentrated in the balls $\bigcup_k B_1(x_k)$ manage to
screen the nucleus charge already outside of the radius
$\frac{2}{3}L$, while we know that there are still charges at radius
$L$.  This will be used in Lemma~\ref{lem:whOmega}.
\begin{proof}
  For any $y \in \bigcup_k B_1(x_k)$, we have
  \begin{equation*}
    \abs{x - y} \leq \abs{x} + \abs{y} \leq \abs{x} + \max_k \abs{x_k} + 1
    \leq \abs{x} + L + 1,
  \end{equation*}
  and hence
  \begin{equation*}
    \int_{\Omega \cap \;\bigcup_k B_1(x_k)} \frac{1}{\abs{x - y}} \ud y
    \geq \frac{1}{\abs{x} + L + 1} \, \Bigl\lvert \Omega \cap \;\bigcup_k B_1(x_k)
    \Bigr\rvert  \stackrel{\eqref{eq:chargeinballs}}{\geq} \frac{13 c_0 Z_{\eff}}{\abs{x} + L + 1}.
  \end{equation*}
  Combining with the inequality \eqref{eq:screenrz} (note that
  $\tfrac{2}{3}L \geq \tfrac{4}{3} R_Z > R_Z$ and hence
  $B_{R_Z} \subset B_{\frac{2}{3}L}$), we just need to show that
  \begin{equation*}
    \frac{13}{\abs{x} + L + 1} > \frac{1}{\abs{x} - R_Z}, 
  \end{equation*}
  which is equivalent to $12 \abs{x} - L > 13 R_Z + 1$. The lemma
  follows by observing that $12 \abs{x} \geq 8 L$ and $7L \geq 7 R \geq
  14 R_Z > 13 R_Z + 1$ as $R_Z \sim Z^{1/3} \gg 1$. 
\end{proof}

The proof of Proposition~\ref{prop:refined} relies on the
connectedness of the minimizing set $\Omega$ where the nucleus charge
is screened. In connection with Lemma~\ref{lem:screenBall} this
implies that outside of the radius $\frac{2}{3}L$ and away from the
points $\{x_k\}_{k =1, \ldots, N}$, any point in $\Omega$ must be
connected to the rest of $\Omega$.
\begin{lemma}[Connectedness]\label{lem:connect}
  Define the set $\wh{\Omega}$ as
  \begin{equation*}
    \wh{\Omega} := 
    \Omega - ( B_{\tfrac{2}{3} L + H} \cup \bigcup_{k=1}^N B_H(x_k) ),
  \end{equation*}
  where $1 < H \ll Z^{1/3}$ is any universal constant (to be fixed
  later).  Then for any $x \in \wh{\Omega}$ and any radius $r$ that
  $0 < r < H - 1$, we have
  \begin{equation*}
    \bigl\lvert \partial B_r(x) \cap \Omega \bigr\rvert > 0.
  \end{equation*}
\end{lemma}

\begin{proof}
  Suppose for some $x \in \wh{\Omega}$ and some $r$
  \begin{equation}\label{eq:noboundary}
    \bigl\lvert \partial B_r(x) \cap \Omega \bigr\rvert = 0. 
  \end{equation}
  By comparing the energy of the minimizer $\Omega$ with the set that
  consists of $\Omega - B_r(x)$ and a translation of $\Omega
  \cap B_r(x)$ far away so that the two pieces are well separated, we
  obtain by minimality 
  \begin{equation}\label{eq:ec1}
    \mf{E}_Z(\Omega) \leq \mf{E}_Z\bigl(\Omega - B_r(x)\bigr)
    + \mf{E}_0\bigl( \Omega \cap B_r(x) \bigr). 
  \end{equation}
  
  We now show that \eqref{eq:ec1} cannot be true since the new
  configuration has the same interfacial energy thanks to
  \eqref{eq:noboundary} and strictly less electrostatic energy, as we
  shall presently argue.  Note that for $y \in B_r(x)$, thanks to
  $r < H - 1$
  \begin{equation*}
    \abs{y} \geq \abs{x} - \abs{y - x} \geq \tfrac{2}{3} L + H - r > \tfrac{2}{3} L. 
  \end{equation*}
  Therefore,  
  \begin{equation*}\label{eq:Brcontain}
    B_r(x) \cap  B_{\frac{2}{3}L} = \emptyset, 
  \end{equation*}
  and hence \eqref{eq:screenBall} in Lemma~\ref{lem:screenBall} holds
  on $B_r(x)$.
  Also for each $k$, we have $\dist(x, B_1(x_k)) \geq H - 1 > r$, and thus
  \begin{equation*}
    B_r(x) \cap B_1(x_k) = \emptyset.
  \end{equation*}
  The last two statements combine to 
  \begin{equation}\label{eq:Brcontain2}
    \Omega - B_r(x) \supset \Omega \cap \bigl( B_{\frac{2}{3}L} \cup 
    \bigcup_{k=1}^N B_1(x_k)\bigr).  
  \end{equation}
  Finally, we get
  \begin{multline*}
    \mf{E}_Z(\Omega) - \mf{E}_Z\bigl(\Omega - B_r(x)\bigr) -
    \mf{E}_0\bigl( \Omega \cap B_r(x) \bigr) = \int_{\Omega \cap
      B_r(x)} \int_{\Omega - B_r(x)} \frac{1}{\abs{y - y'}}
    \ud y' \ud y
    - \int_{\Omega \cap B_r(x)} \frac{Z}{\abs{y}} \ud y \\
    \stackrel{\eqref{eq:Brcontain2}}{\geq} \int_{\Omega\cap B_r(x)} 
    \left( \int_{\Omega \cap (B_{\frac{2}{3}L} \cup \;\bigcup_{k} B_1(x_k))}
      \frac{1}{\abs{y-y'}} \ud y' - \frac{Z}{\abs{y}} \right) \ud y \stackrel{\eqref{eq:screenBall}}{>} 0.
  \end{multline*}
  The contradiction with \eqref{eq:ec1} concludes the proof. 
\end{proof}

Lemma~\ref{lem:whOmega} now shows that because there is a point in $\Omega$ at distance $L$ (namely $x_N$) and in view of the connectedness from Lemma~\ref{lem:connect}, there is a substantial amount of excess volume outside of $\frac{2}{3} L$ and away from the $\{x_k\}_{k = 1, \ldots, N}$. 
\begin{lemma}\label{lem:whOmega} We have for 
  $\wh{\Omega}$ defined in Lemma~\ref{lem:connect}
  \begin{equation*}
    \bigl\lvert \wh{\Omega} \bigr\rvert 
    \gtrsim V-Z. 
  \end{equation*}
\end{lemma}
\begin{proof} By definition, 
  \begin{equation*}
      \bigl\lvert \wh{\Omega} \bigr\rvert 
      \geq \abs{\Omega} - \abs{\Omega \cap B_{\tfrac{2}{3} L  + H}} - \sum_{k=1}^N
      \abs{B_H(x_k)} 
      = V - \abs{\Omega \cap B_{\tfrac{2}{3} L  + H}}  - \sum_{k=1}^N \abs{B_H(x_k)}.
  \end{equation*}
  Since $N \leq 10 c Z_{\eff} / \delta
  + 1 \lesssim Z_{\eff}$, we have   
  \begin{equation*}
    \sum_{k=1}^N \abs{B_H(x_k)} \lesssim H^3 Z_{\eff}. 
  \end{equation*}
  Thus, as by assumption $ V - Z \gg Z_{\eff}$ and $H$ being a universal constant,
  the conclusion of the lemma follows from 
  \begin{equation*} 
    \abs{\Omega \cap B_{\tfrac{2}{3} L  + H}} - Z \lesssim Z_{\eff}. 
  \end{equation*}
  Since $\abs{\Omega \cap B_R} - Z \lesssim Z_{\eff}$ from Lemma~\ref{lem:screenC}, it suffices to consider the case $\tfrac{2}{3}L + H > R$ and to show that 
  \begin{equation}\label{eq:OmegaLH}
    \abs{\Omega \cap A_{R, \tfrac{2}{3} L + H}} \lesssim Z_{\eff}.
  \end{equation}
  
  We now claim that by the choice of $\{x_k\}$,
  \begin{equation}\label{eq:OmegaAnnulus}
    \Omega \cap A_{R, \tfrac{2}{3} L + H} \subset \bigcup_{k=1}^N B_2(x_k). 
  \end{equation}
  Indeed, if there exists $x \in A_{R, \tfrac{2}{3} L + H}$ such that $x \in
  \Omega$ in the sense of \eqref{eq:belongset} and $x \not \in
  \bigcup_k B_2(x_k)$, we have $\abs{x - x_k} \geq 2$ and also
  \begin{equation*}
    R \leq \abs{x} \leq \frac{2}{3} L + H < L = \max_{1\leq k \leq N} \abs{x_k}.
  \end{equation*}
  thanks to
  $\tfrac{1}{3} L \geq \tfrac{1}{3} R \gtrsim R_Z \sim Z^{1/3} \gg H$.
  Therefore $\abs{x} < \abs{x_k}$ for some $k$, which violates the
  definition of $x_k$.  \eqref{eq:OmegaLH} now follows from
  \eqref{eq:OmegaAnnulus}, since $N \lesssim Z_{\eff}$.
\end{proof}

We now use once more the connectedness to argue that the electrostatic
energy coming from the charge outside of $\frac{2}{3}L$ and away from
the $\{x_k\}$'s is substantial.
\begin{lemma}\label{lem:CoulombEnergy}
  We have for the Coulomb energy of the set $\wh{\Omega}$ defined in
  Lemma~\ref{lem:connect}
  \begin{equation}\label{eq:lemconnect}
    \frac{1}{2} \int_{\wh{\Omega}}\int_{\Omega - (B_{\frac{2}{3}L} \cup \;\bigcup_k B_1(x_k))} 
    \frac{1}{\abs{x-y}} \ud y\ud x
    \;\gtrsim\; \bigl\lvert
    \wh{\Omega}\bigr\rvert \ln H 
    \gtrsim (V-Z) \ln H.
  \end{equation}
\end{lemma}

\begin{proof}
  By Lemma~\ref{lem:whOmega}, it suffices to show that for any $x
  \in \wh{\Omega}$ 
  \begin{equation*}
    \frac{1}{2}\int_{\Omega - (B_{\frac{2}{3}L} \cup \;\bigcup_k B_1(x_k))} 
    \frac{1}{\abs{x - y}} \ud y 
    \gtrsim \ln H.  
  \end{equation*}
  We take the largest integer $M$ such that 
  \begin{equation}\label{eq:MHbound}
    2 M \leq  H - 2.
  \end{equation}
  By Lemma~\ref{lem:connect}, there exist points $y_1, y_2, \cdots,
  y_M$ such that for $i = 1, \cdots, M$ 
  \begin{equation*}
    y_i \in \partial B_{2i}(x) \quad \text{and} \quad 
    y_i \in \Omega \text{ in the sense of \eqref{eq:belongset}}.
  \end{equation*}
  Indeed, we have for the inner trace $\chi^i$ of the characteristic
  function $\chi$ of $\Omega$ that
  $\int_{\partial B_{2i}} \chi^i \ud \mc{H}^2 > 0$. Take an
  $\ud \mc{H}^2 \vert_{\partial B_{2i}}$-Lebesgue point $y_i$ of
  $\chi^i$ with $\chi^i(y_i) = 1$. Such a choice satisfies the
  conditions. See Figure~\ref{fig:connect} for an illustration.
  \begin{figure}[ht]
    \includegraphics[width = 0.95\textwidth]{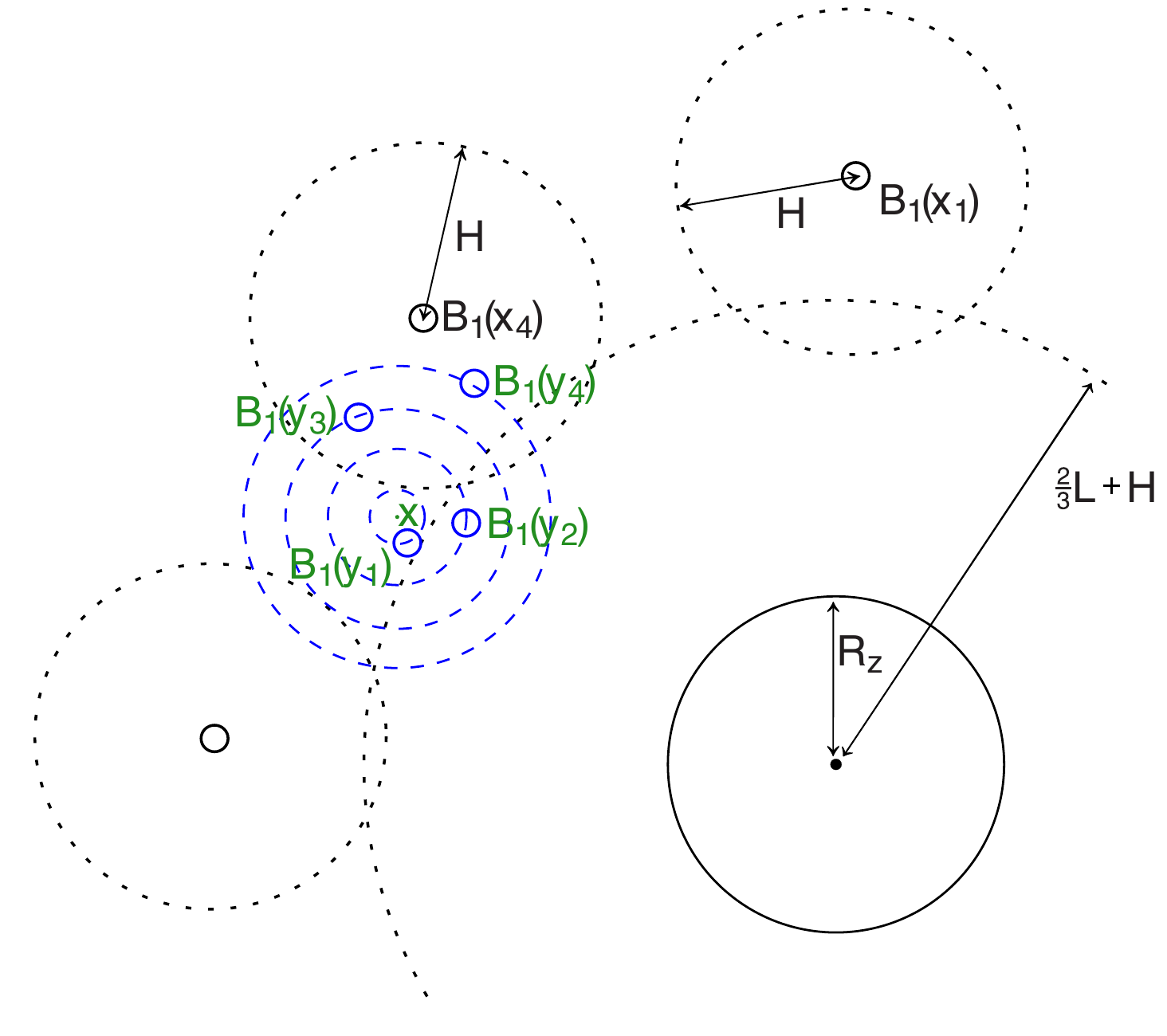}
    \caption{A point $x \in \wh{\Omega}$ and the sequence of points
      $\{y_1, \ldots, y_M\}$ as in the proof of
      Lemma~\ref{lem:CoulombEnergy}. \label{fig:connect}}
  \end{figure}

  We note that for fixed $i = 1, \ldots, M$ we have 
  \begin{equation*}\label{eq:absy}
    \abs{y_i} \geq \abs{x} - \abs{y_i - x} \geq \tfrac{2}{3} L + H  - 2M
    \stackrel{\eqref{eq:MHbound}}{>} \tfrac{2}{3} L + 1, 
  \end{equation*}
  and hence 
  \begin{equation}\label{eq:23L}
    B_1(y_i) \cap B_{\frac{2}{3} L} = \emptyset, 
  \end{equation}
  so that \eqref{eq:screenBall} in Lemma~\ref{lem:screenBall} holds on
  $B_1(y_i)$. 
  Moreover, because of 
  \begin{equation*}
    \abs{y_i - x_k} \geq \abs{x - x_k} - \abs{y_i - x} \geq H - 2M \stackrel{\eqref{eq:MHbound}}{\geq} 2
  \end{equation*}
  for any $k = 1, \ldots, N$ we have 
  \begin{equation*}
    B_1(y_i) \cap  \bigcup_{k=1}^N B_1(x_k) = \emptyset.
  \end{equation*}
  Combined with \eqref{eq:23L}, this gives
  \begin{equation}\label{eq:ballcontained}
    B_1(y_i) \cap \bigl(B_{\frac{2}{3}L} \cup  \bigcup_{k=1}^N B_1(x_k)\bigr) = \emptyset. 
  \end{equation}
  Therefore, \eqref{eq:screenBall} yields in particular 
  \begin{equation*}
    \phi_{y_i}(y') = \int_{\Omega - B_1(y_i)} \frac{1}{\abs{y' - y}} \ud y - \frac{Z}{\abs{y'}} > 0
    \qquad \text{for} \quad y' \in B_1(y_i).
  \end{equation*}
  Hence we may apply Lemma~\ref{lem:density} to get 
  \begin{equation}\label{eq:massball}
    \abs{B_1(y_i) \cap \Omega} \gtrsim 1. 
  \end{equation}
  Finally, by construction, $\{ B_1(y_i)\}_{i =1, \ldots, M}$ are
  pairwise disjoint so that for our fixed $x \in \wh{\Omega}$
  \begin{equation*}
    \begin{aligned}
      \frac{1}{2}\int_{\Omega - (B_{\frac{2}{3}L} \cup \; \bigcup_k B_1(x_k))}
      \frac{1}{\abs{x - y}} \ud y &
      \stackrel{\eqref{eq:ballcontained}}{\geq} \frac{1}{2} \sum_{i =
        1}^M \int_{\Omega \cap B_1(y_i)}
      \frac{1}{\abs{x - y}} \ud y\\
      & \stackrel{\eqref{eq:massball}}{\gtrsim} \sum_{i = 1}^M
      \frac{1}{2i + 1} \sim \ln M.
    \end{aligned}
  \end{equation*}
  We conclude by returning to the choice of $M$.
\end{proof}

\bigskip

We are now ready to prove Proposition~\ref{prop:refined}. 
\begin{proof}[Proof of Proposition~\ref{prop:refined}]
  Suppose $\Omega$ is a minimizer with $V - Z \gg
  Z_{\eff}$. 
  First note that
  \begin{equation*}
    \int_{\tfrac{2}{3}L}^{\tfrac{2}{3} L + H} \abs{\partial B_r \cap \Omega} \ud r = 
    \abs{\Omega \cap A_{\tfrac{2}{3}L, \tfrac{2}{3} L + H}} \leq 
    \abs{\Omega \cap A_{R_Z, \tfrac{2}{3} L + H}} 
    = \abs{\Omega \cap B_{\tfrac{2}{3} L + H} } - \abs{\Omega \cap B_{R_Z}}
    \lesssim Z_{\eff},
  \end{equation*}
  where in the last inequality, we have used \eqref{eq:OmegaLH} and
  \eqref{eq:surpluscharge}.  Therefore, there exists a $r_0 \in
  [\tfrac{2}{3}L, \tfrac{2}{3} L + H]$ such that
  \begin{equation}\label{eq:boundaryr0}
    \abs{\partial B_{r_0} \cap \Omega} \leq \slashint_{\tfrac{2}{3}L}^{\tfrac{2}{3} L + H } 
    \abs{\partial B_r \cap \Omega} \ud r \lesssim \frac{1}{H} Z_{\eff} 
    \leq Z_{\eff}. 
  \end{equation}

  Let us consider the comparison set
  $\wb{\Omega} = \Omega \cap ( B_{r_0} \cup \bigcup_{k=1}^N
  B_1(x_k))$.
  We have an upper bound for the energy of the set $\Omega$ (as in
  Lemma~\ref{lem:Eupperbound}):
  \begin{equation*}
    \mf{E}_Z(\Omega) -  \mf{E}_Z(\wb{\Omega}) \leq E_0(\abs{\Omega
      - \wb{\Omega}}) \lesssim \abs{\Omega
      - \wb{\Omega}} + 1. 
  \end{equation*}
  Using Proposition~\ref{prop:closeball}, we have 
  \begin{equation*}
    \abs{\Omega -\wb{\Omega}}
    \leq \abs{\Omega - B_{R_Z}}
    = \abs{\Omega} - \abs{B_{R_Z}} + \abs{B_{R_Z} - \Omega}
    \leq V - Z + C Z_{\eff} \lesssim V - Z. 
  \end{equation*}
  The last two estimates combine to, as $V - Z \gg 1$,
  \begin{equation}\label{eq:EZupper}
    \mf{E}_Z(\Omega) -  \mf{E}_Z(\wb{\Omega}) \lesssim V - Z. 
  \end{equation}
  On the other hand, 
  \begin{equation}\label{eq:EZlower}
      \begin{aligned}
        \mf{E}_Z(\Omega) - \mf{E}_Z(\wb{\Omega}) & \geq -
        \abs{\partial
          (B_{r_0}\cup  \cup_k B_1(x_k)) \cap \Omega} \\
        & \qquad + \int_{\Omega - \wb{\Omega}} \ud x \left(
          \int_{\wb{\Omega}} \frac{1}{\abs{x-y}} \ud y -
          \frac{Z}{\abs{x}} \right) + \frac{1}{2}
        \int_{\Omega-\wb{\Omega}} \int_{\Omega-\wb{\Omega}}
        \frac{1}{\abs{x - y}} \ud x
        \ud y \\
        & \geq - \abs{\partial B_{r_0} \cap \Omega}
        - \sum_{k=1}^N \abs{\partial B_1(x_k)} \\
        & \qquad + \int_{\Omega - \wb{\Omega}} \ud x \left(
          \int_{\wb{\Omega} \cap (B_{\frac{2}{3}L} \cup\;\bigcup_k
            B_1(x_k))} \frac{1}{\abs{x-y}} \ud y -
          \frac{Z}{\abs{x}} \right) \\
        & \qquad + \int_{\Omega - \wb{\Omega}} \int_{\wb{\Omega} -
          (B_{\frac{2}{3}L} \cup \;\bigcup_k B_1(x_k))}
        \frac{1}{\abs{x-y}} \ud y \ud x + \frac{1}{2} \int_{\Omega-
          \wb{\Omega}} \int_{\Omega- \wb{\Omega}} \frac{1}{\abs{x -
            y}} \ud x \ud y \\
        & \geq - C Z_{\eff} + \int_{\Omega - \wb{\Omega}} \ud x \left(
          \int_{\wb{\Omega} \cap (B_{\frac{2}{3}L} \cup\;\bigcup_k
            B_1(x_k))} \frac{1}{\abs{x-y}} \ud y -
          \frac{Z}{\abs{x}} \right) \\
        & \qquad + \frac{1}{2} \int_{\Omega - \wb{\Omega}} \int_{\wb{\Omega} -
          (B_{\frac{2}{3}L} \cup \;\bigcup_k B_1(x_k))}
        \frac{1}{\abs{x-y}} \ud y \ud x + \frac{1}{2} \int_{\Omega-
          \wb{\Omega}} \int_{\Omega- \wb{\Omega}} \frac{1}{\abs{x -
            y}} \ud x \ud y,
    \end{aligned}
  \end{equation}
  where we have used \eqref{eq:boundaryr0} and that
  $N \lesssim Z_{\eff}$ to obtain the last inequality.  To further
  estimate the right hand side, we observe that since
  $r_0 \geq \frac{2}{3}L$
  \begin{equation*}
    \wb{\Omega} \cap \Bigl(B_{\frac{2}{3}L} \cup\;\bigcup_k
    B_1(x_k)\Bigr) = (\Omega \cap B_{r_0}) \cap \Bigl(B_{\frac{2}{3}L} \cup\;\bigcup_k
    B_1(x_k)\Bigr) = \Omega \cap \Bigl(B_{\frac{2}{3}L} \cup\;\bigcup_k
    B_1(x_k)\Bigr), 
  \end{equation*}
  and also 
  \begin{equation*}
    \Omega - \wb{\Omega} = \Omega - \Bigl(B_{r_0} \cup
  \; \bigcup_k B_1(x_k)\Bigr) \subset \bigl(B_{\frac{2}{3}L}\bigr)^c.
  \end{equation*}
  Therefore, using Lemma~\ref{lem:screenBall}, the second term on the
  right hand side of \eqref{eq:EZlower} is non-negative. Thus, we
  obtain
  \begin{equation*}
    \mf{E}_Z(\Omega) - \mf{E}_Z(\wb{\Omega}) \geq - C Z_{\eff} 
    + \frac{1}{2} \int_{\Omega - \wb{\Omega}} \int_{\Omega - (B_{\frac{2}{3}L} \cup \;\bigcup_k B_1(x_k))} \frac{1}{\abs{x - y}} \ud y \ud x,
  \end{equation*}
  where the last two terms on the right hand side of
  \eqref{eq:EZlower} have been combined. Since
  $r_0 \leq \frac{2}{3}L + H$, we have
  $\wh{\Omega} \subset \Omega - \wb{\Omega} = \Omega - (B_{r_0} \cup
  \; \bigcup_k B_1(x_k))$, and hence 
  \begin{equation*}
    \mf{E}_Z(\Omega) - \mf{E}_Z(\wb{\Omega}) \geq - C Z_{\eff} 
    + \frac{1}{2} \int_{\wh{\Omega}} \int_{\Omega - (B_{\frac{2}{3}L} \cup \;\bigcup_k B_1(x_k))} \frac{1}{\abs{x - y}} \ud y \ud x
    \stackrel{\eqref{eq:lemconnect}}{\geq} - C Z_{\eff} + \frac{1}{C} (V - Z) \ln H.
  \end{equation*}
  Combining with the upper bound \eqref{eq:EZupper}, we obtain
  \begin{equation*}
    \frac{1}{C} (V - Z) \ln H  - C Z_{\eff} \leq V - Z, 
  \end{equation*}
  or equivalently, as $Z_{\eff} \ll (V - Z)$, $(V - Z) \ln H \lesssim
  V - Z$.  Hence, we get a contradiction by choosing $H$ sufficiently
  large. This concludes the proof.
\end{proof}

\section{Proof of the Density Lemma~\ref{lem:density}} 
\label{sec:density}

We now prove the Density Lemma~\ref{lem:density}, and so complete
the proof of Theorem~\ref{thm:nonexist}.  Let us recall and
reformulate the set up first. For simplicity of notation, up to a
translation, we may assume that the ball considered in
Lemma~\ref{lem:density} is centered at the origin. Define the energy
functional
\begin{equation}\label{eq:mcE}
  \mc{E}_{\phi}(\Omega) = \abs{\partial \Omega \cap  B} + \int_B \phi \chi_{\Omega} + \frac{1}{2} \iint_{B \times B} \frac{\chi_{\Omega}(x) \chi_{\Omega}(y)}{\abs{x - y}} \ud x \ud y
\end{equation}
such that $\phi > 0$ and $\Delta \phi = 0$ in $B$ (note that by our
assumption the nucleus is outside the ball). This is the local (in
$B$) contribution to the total energy when the set $\Omega$ is fixed
outside the ball. The optimality of $\Omega$ then implies
\begin{equation}\label{eq:mcEopt}
  \begin{aligned}
    & \mc{E}_{\phi}(\Omega) \leq \mc{E}_{\phi}(\Omega') + E_0\bigl(\abs{\Omega} - \abs{\Omega'}\bigr) \\
    & \text{for all } \Omega' \text{ such that } \abs{\Omega'} \leq
    \abs{\Omega} \text{ and } \Omega' \Delta \Omega \subset\subset B.
  \end{aligned}
\end{equation}
Here $\Omega' \Delta \Omega = (\Omega' - \Omega) \cup (\Omega -
\Omega')$ is the symmetric difference between two sets.  Note that we
do not require that the volume $\abs{\Omega' \cap B}$ equals to
$\abs{\Omega \cap B}$ as we have the freedom to remove part of the
mass from the ball $B$ and put it far away from $\Omega$ and pay
energy $E_0(\abs{\Omega} - \abs{\Omega'})$ (as in
Lemma~\ref{lem:Eupperbound}).  Our goal is to show that if the origin
is in the optimal set $\Omega$ in the measure-theoretic sense of
\eqref{eq:belongset}, the volume $\abs{\Omega \cap B}$ is at least a
universal constant $\delta > 0$.

\smallskip

Let us give a sketch of the proof here.  The first step is to reduce
by dilating the ball to $B_{R}$ to the case that
$\sup \abs{\nabla \phi}$ is small and the volume inside is $1$ as in
Section~\ref{sec:scaling}. Then we consider the problem on a slightly
smaller ball $B_{R_0}$ for $R_0 \in [R-1, R]$, so that we may assume
that the intersection of $\Omega$ and the boundary of that ball
$\abs{\Omega \cap \partial B_{R_0}}$ is at most $1$. The proof is
completed by considering separately the case that $\Omega$ has a
negligible volume inside $B_{R_0}$ (Section~\ref{sec:boundary}) and
the case that $\abs{\Omega \cap B_{R_0}}$ consists of a large fraction
of $\abs{\Omega}$ (Section~\ref{sec:interior}). In both cases, we
compare the minimizer with sets obtained by local deformation inside
the ball.

\subsection{Scaling argument}\label{sec:scaling}

Let us start with the following lemma which states that the potential
$\phi$ cannot be too small if $\abs{\nabla \phi}$ is large.
\begin{lemma}\label{lem:intphi}
  Let $\phi$ satisfy the equation 
  \begin{equation}
    - \Delta \phi = 4\pi \chi_{\Omega} \quad \text{in } B_2 
  \end{equation}
  for some set $\Omega$ and assume that $\phi \geq 0$ on
  $B_1$. We then have
  \begin{equation}\label{eq:intphi}
    \sup_{B_1} \abs{\nabla \phi} \lesssim \inf_{B_1} \phi + 1. 
  \end{equation}
\end{lemma}
\begin{proof}
  Define $\psi$ as 
  \begin{equation*}
    \psi(x) = \int_{B_2} \frac{1}{\abs{x - y}} \chi_{\Omega}(y) \ud y. 
  \end{equation*}
  Then $h = \phi - \psi + 2\pi$ is harmonic in $B_2$ and non-negative since 
  \begin{equation*}
    \psi(x) \leq \int_{B_2} \frac{1}{\abs{x - y}} \ud y 
    \leq \int_{B_2} \frac{1}{\abs{y}} \ud y = 2 \pi. 
  \end{equation*} 
  Using interior regularity for $h$ in terms of 
  \begin{equation*}
    \sup_{B_1} \abs{\nabla h} \lesssim \sup_{B_{3/2}} h
  \end{equation*}
  and then the Harnack's inequality, \textit{i.e.}, 
  \begin{equation*}
    \sup_{B_{3/2}} h \lesssim  \inf_{B_{3/2}} h \lesssim \inf_{B_1} h,  
  \end{equation*}
  we obtain by definition of $h$ 
  \begin{equation*}
    \begin{aligned}
      \sup_{B_1} \abs{\nabla \phi} & \leq \sup_{B_1} \abs{\nabla h}
      + \sup_{B_1} \abs{\nabla \psi}  \\
      & \lesssim \inf_{B_1} h + \sup_{B_1} \abs{\nabla \psi} \\
      & \leq \inf_{B_1} \phi + 2 \pi + \sup_{B_1}\abs{\nabla \psi} \\
      & \leq \inf_{B_1} \phi + C.
    \end{aligned}
  \end{equation*}
\end{proof}



Defining for sets $\Omega \subset B_R$
\begin{equation}\label{eq:defER}
  \mc{E}_{\phi, R}(\Omega) = \abs{\partial \Omega \cap B_{R}} + \int_{B_{R}} \phi \chi_{\Omega} + \frac{1}{R^3} \frac{1}{2} \iint_{B_{R} \times B_{R}} 
  \frac{\chi_{\Omega}(x) \chi_{\Omega}(y)}{\abs{x - y}} \ud x \ud y,
\end{equation}
we now show that it suffices to consider the following equivalent
formulation of Lemma~\ref{lem:density}.
\begin{lemma}\label{lem:moddensity}
  Let $\Omega$ be a set with $0 \in \Omega$ (in the sense of \eqref{eq:belongset})
  and $\abs{\Omega \cap B_R} = 1$ that satisfies 
  \begin{equation*}
    \begin{aligned}
      & \mc{E}_{\phi, R}(\Omega) \leq \mc{E}_{\phi, R}(\Omega') + R^2 E_0\bigl(R^{-3} (\abs{\Omega} - \abs{\Omega'})\bigr) \\
      & \text{for all } \Omega' \text{ such that } \abs{\Omega'} \leq
      \abs{\Omega} \text{ and } \Omega' \Delta \Omega \subset\subset B_R, 
    \end{aligned}
  \end{equation*}
  where $\phi > 0$, $\Delta \phi = 0$ and $\sup_{B_{R}}
  \abs{\nabla \phi} \leq 1 / R$. Then we have $R \lesssim 1$. 
\end{lemma}

We prove Lemma~\ref{lem:density} assuming Lemma~\ref{lem:moddensity},
which in turn will be proved in the remaining of this section.
\begin{proof}[Proof of Lemma~\ref{lem:density}]
  Let $M$ be a large universal constant that we will fix later, we
  will proceed in two cases depending whether $\sup_{B_1(x)} \abs{\nabla
    \phi_x} \leq M$ or not, where we recall that
  \begin{equation*}
    \phi_x(y) = \int_{\Omega - B_1(x)} \frac{1}{\abs{y - z}} \ud z - \frac{Z}{\abs{y}}.  
  \end{equation*}
  Without loss of generality, we may assume $x = 0$ (by a translation)
  and $\abs{\Omega \cap B_1} \leq 1$, we also simplify the notation
  and denote $\phi = \phi_x$.

  Case 1. $\sup_{B_1} \abs{\nabla \phi} \leq M$. In this case, we will
  dilate the ball $B_1$ to a ball with large radius so that the
  dilated set  has volume $1$ inside. For this, we use a
  change of variable $x \mapsto R x$, with $R^3 = 1 / \abs{\Omega \cap
    B_1}$, and define
  \begin{align*}
    & \wt{\Omega} = \{ x \mid x / R \in \Omega \}; \\
    & \wt{\phi}(x) = \frac{1}{R} \phi(x / R).
  \end{align*}
  We have then 
  \begin{equation*}
    \mc{E}_{\phi}(\Omega) = \frac{1}{R^2} \abs{\partial \Omega \cap B_{R}} + \frac{1}{R^2} \int_{B_{R}} \wt{\phi} \chi_{\wt{\Omega}} + \frac{1}{R^5} \frac{1}{2} \iint_{B_{R} \times B_{R}} \frac{\chi_{\wt{\Omega}}(x) \chi_{\wt{\Omega}}(y)}{\abs{ x- y}} \ud x \ud y= \frac{1}{R^2} \mc{E}_{\wt{\phi}, R}(\wt{\Omega}),
  \end{equation*}
  by the definition of $\mc{E}_{\phi, R}$ in \eqref{eq:defER}.  Also
  it is easy to see that
  \begin{equation*}
    \sup_{B_{R}} \abs{\nabla \wt{\phi}} \leq M / R^2. 
  \end{equation*}
  Without loss of generality, we may assume that $ M / R \leq 1$ (as
  otherwise, $\abs{\Omega \cap B_1} = 1 / R^3 \geq M^{-3}$, which
  readily implies Lemma~\ref{lem:density}). Hence, we get $\sup_{B_R}
  \abs{\nabla \wt{\phi}} \leq 1/ R$. Applying
  Lemma~\ref{lem:moddensity} to $\wt{\Omega}$ and $\wt{\phi}$, we
  obtain that $R \lesssim 1$, and hence
  \begin{equation*}
    \abs{\Omega \cap B_1} \gtrsim 1. 
  \end{equation*}
  
  Case 2. $\sup_{B_1} \abs{\nabla \phi} > M$. Lemma~\ref{lem:intphi}
  then implies, by taking $M$ sufficiently large, that 
  \begin{equation}\label{eq:defc0}
    C_0 = \inf_{B_1} \phi \gtrsim  
    \sup_{B_1} \abs{\nabla \phi} - 1  \gg 1. 
  \end{equation}
  For $r \in [0, 1)$, consider the competitor set $\Omega' = \Omega -
  B_r$; we have (as in Lemma~\ref{lem:Eupperbound})
  \begin{equation*}
    \mc{E}_{\phi}(\Omega) \leq \mc{E}_{\phi}(\Omega') + E_0\bigl( 
    \abs{\Omega} - \abs{\Omega'} \bigr), 
  \end{equation*}
  which implies that 
  \begin{equation*}
    \abs{\partial(\Omega \cap B_r)} + \abs{\Omega \cap B_r} \inf_{B_1} \phi  + \frac{1}{2} \iint_{(\Omega \cap B_r) \times (\Omega\cap B_r)}
    \frac{1}{\abs{x - y}} \ud x \ud y \\
    \leq 2 \abs{\Omega \cap \partial B_r} + E_0(\abs{\Omega \cap B_r}).
  \end{equation*}
  Since $E_0(\abs{\Omega \cap B_r}) \leq \mf{E}_0(\Omega \cap B_r)$,
  we arrive at 
  \begin{equation*}
    \abs{\Omega \cap B_r} \leq 
    \frac{2}{\inf_{B_1} \phi}  \, \abs{\Omega \cap \partial B_r}
    =
    \frac{2}{C_0}  \, \abs{\Omega \cap \partial B_r}. 
  \end{equation*}
  Using $\abs{\Omega \cap \partial B_r} = \frac{\rd}{\rd r}
  \abs{\Omega \cap B_r}$, this yields 
  \begin{equation*}
    \abs{\Omega \cap B_r} \leq \exp\Bigl( - \frac{1-r}{2} C_0 \Bigr) \abs{\Omega \cap B_1}. 
  \end{equation*}
  Together with $\abs{\Omega \cap B_1} \leq 1$, we obtain 
  \begin{equation}\label{eq:expcontrol}
    \abs{\Omega \cap B_{\frac{1}{2}}} \leq \exp\Bigl(- \frac{1}{4} C_0\Bigr). 
  \end{equation}
  Since by assumption, $0 \in \Omega$ in the sense of
  \eqref{eq:belongset}, we have
  $\abs{\Omega \cap B_{\frac{1}{2}}} > 0$, so that we may as in Case 1
  dilate the ball $B_{\frac{1}{2}}$ to $B_R$ such that
  $\abs{\wt{\Omega} \cap B_R} = 1$. By \eqref{eq:expcontrol} we have
  \begin{equation*}
    R \geq \exp\Bigl(\frac{1}{12} C_0 \Bigr). 
  \end{equation*}
  For the dilated potential, we have
  \begin{equation*}
    \sup_{B_R} \abs{\nabla \wt{\phi}} \leq \frac{1}{R^2} 
    \sup_{B_1} \abs{\nabla \phi} \stackrel{\eqref{eq:defc0}}{\lesssim} 
    \frac{1}{R^2} C_0 
    \leq \frac{1}{R} C_0 \exp\Bigl(- \frac{1}{12} C_0 \Bigr), 
  \end{equation*}
  so that Lemma~\ref{lem:moddensity} is applicable for $C_0 \gg 1$ (\ie, $M \gg 1$), yielding $R\lesssim 1$ and thus $\abs{\Omega \cap
    B_{\frac{1}{2}}} \gtrsim 1$.
\end{proof}

To prove Lemma~\ref{lem:moddensity}, first observe that, while we do
not have a priori control on $\abs{\Omega \cap \partial B_{R}}$, we
have
\begin{equation*}
  \int_{R-1}^{R} \abs{\Omega \cap \partial B_r} \ud r 
  = \abs{\Omega \cap A_{R - 1, R}} \leq 1. 
\end{equation*}
It results that, there exists $R_0 \in [R-1, R]$ such that
$\abs{\Omega \cap \partial B_{R_0}} \leq 1$. We will proceed with the
proof of Lemma~\ref{lem:moddensity} in two cases, depending on how
much volume the set $\Omega \cap B_{R_0}$ has
\begin{itemize}
\item[Case 1.] $\abs{\Omega \cap B_{R_0}} \leq 1/2$, studied in
  Section~\ref{sec:boundary}, and
\item[Case 2.] $\abs{\Omega \cap B_{R_0}} \geq 1/2$, studied in
  Section~\ref{sec:interior}.
\end{itemize}
Without loss of generality, we will always assume $R \gg 1$ as
otherwise, there is nothing to prove.

\subsection{Proof of Lemma~\ref{lem:moddensity}: Case $\abs{\Omega  \cap B_{R_0}} \leq 1/2$}
\label{sec:boundary}

Since $\abs{\Omega \cap B_{R}} = 1$, we have $\abs{\Omega \cap
  A_{R_0, R}} \geq 1/2$.

\begin{figure}[ht]
\includegraphics[width = 0.95\textwidth]{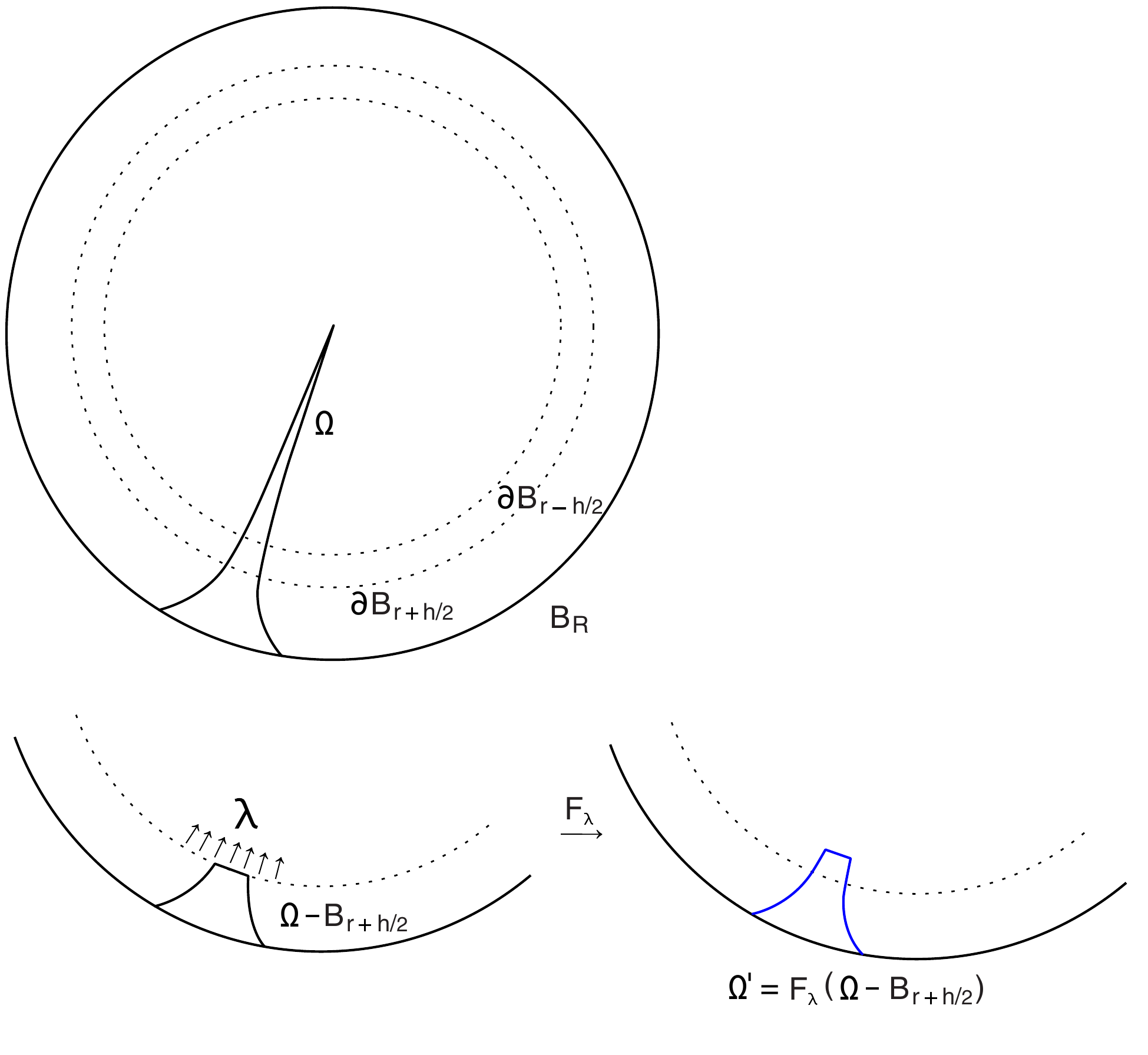}
\caption{The set $\Omega$ in $B_R$ and the deformation $F_{\lambda}$ acting on $\Omega - B_{r+h/2}$ as in the proof of Lemma~\ref{lem:moddensity}. \label{fig:deform}}
\end{figure}

In this case, we will use a deformation $F_{\lambda}$ (see
Figure~\ref{fig:deform}) that stretches the annulus
$A_{\frac{3}{4}R, R}$ into the annulus
$A_{(1 - \lambda)\frac{3}{4}R, R}$ in a radially symmetric way: For a
point $(r, \omega)$ in the spherical coordinate,
$r \in [\tfrac{3}{4} R, R]$ and $\omega \in \mathbb{S}^2$
\begin{equation}
  F_{\lambda}(r, \omega) = \bigl(f_{\lambda}(r), \omega\bigr) := \bigl(r - \lambda (R - r), \omega \bigr). 
\end{equation}
For the set after the deformation, we have the following result. It
states that the deformation increases the volume of the set, without
increasing too much the energy $\ms{E}_{\phi, R}$ of the set.
\begin{prop}\label{prop:deformboundary}
  Let $E \subset A_{\frac{3}{4} R, R}$ be a set of finite
  perimeter, we have for $\lambda \ll 1$
  \begin{align}
    \label{eq:deformbounday1}
    \abs{F_{\lambda}(E)} - \abs{E} & \gtrsim \lambda \abs{E}; \\
    \label{eq:deformbounday2}
    \bigl\lvert \partial F_{\lambda}(E) \cap B_R \bigr\rvert -
    \bigl\lvert \partial E \cap B_R \bigr\rvert
    & \lesssim \lambda \bigl(     \bigl\lvert \partial E \cap B_R \bigr\rvert - \abs{ E \cap \partial B_R}  + 4 R^{-1} \abs{E} \bigr); \\
    \label{eq:deformbounday3} \text{provided } \sup_{B_R} \abs{\nabla \phi} \leq \frac{1}{R}: \qquad \quad
    \int_{F_{\lambda}(E)} \phi - \int_E \phi & \lesssim \lambda \int_E(\phi + 1);
    \intertext{and moreover,}
    \label{eq:deformbounday4}
    \frac{1}{2} \iint_{F_{\lambda}(E) \times F_{\lambda}(E)}
    \frac{1}{\abs{x - y}} \ud x \ud y - \frac{1}{2} \iint_{E \times E}
    \frac{1}{\abs{x - y}} \ud x \ud y & \lesssim \lambda \frac{1}{2}
    \iint_{E \times E} \frac{1}{\abs{x - y}} \ud x \ud y.
  \end{align}
\end{prop}
\begin{remark}
  The appearance of the good term $-\abs{E \cap \partial B_R}$ on the
  r.h.s.{} of \eqref{eq:deformbounday2}, together with the a priori
  estimate in Lemma~\ref{lem:energyapriori} below, will be crucial in
  the proof of Lemma~\ref{lem:moddensity}.
\end{remark}

\begin{proof}
  We clearly have 
  \begin{align}
    \label{eq:F} & \abs{F_{\lambda}(x) - x} \lesssim \lambda R,  \\
    \label{eq:nablaF} & \abs{\nabla F_{\lambda}(x) - \mathrm{Id}} \lesssim \lambda, \quad \text{and thus} \\
    \label{eq:JF} & \abs{J F_{\lambda} - 1} \lesssim \lambda,
  \end{align}
  where $J F_{\lambda}$ denotes the Jacobian of $F_{\lambda}$. Using
  the explicit form of $F_{\lambda}$, we see
  \begin{equation*}
    JF_{\lambda}(x) = (1 + \lambda) \Bigl( \frac{f_{\lambda}(\abs{x})}{\abs{x}}\Bigr)^2 \stackrel{\abs{x} \geq \frac{3}{4}R}{\geq} ( 1 + \lambda)^2 \bigl( 1 - \frac{1}{3} \lambda\bigr)^2
    = 1 + \frac{1}{3}\lambda + \Or(\lambda^2), 
  \end{equation*}
  so that 
  \begin{equation}\label{eq:JFlower}
    JF_{\lambda} - 1 \gtrsim \lambda. 
  \end{equation}

  The estimates \eqref{eq:deformbounday1}, \eqref{eq:deformbounday3},
  and \eqref{eq:deformbounday4} follow from the above properties of
  $F$, as we shall explain now. For the volume, we have
  \begin{equation*}
    \abs{F_{\lambda}(E)} - \abs{E} = \int_E (JF_{\lambda}(x) - 1) \ud x 
    \stackrel{\eqref{eq:JFlower}}{\gtrsim} \lambda \abs{E}. 
  \end{equation*}
  Next for the potential term, we calculate
  \begin{equation*}
    \begin{aligned}
      \int_{F_{\lambda}(E)} \phi - \int_E \phi & = \int_E
      \phi(F_{\lambda}(x))
      J F_{\lambda}(x) - \phi(x) \ud x \\
      & = \int_E \bigl( \phi(F_{\lambda}(x)) - \phi(x) \bigr) J
      F_{\lambda}(x) + \phi(x) ( JF_{\lambda}(x) - 1) \ud x \\
      & \leq \int_E \sup \abs{\nabla \phi} \bigl\lvert F_{\lambda}(x)
      - x\bigr\rvert \abs{J
        F_{\lambda}(x)} + \phi(x) \bigl\lvert JF_{\lambda}(x) - 1\bigr\rvert  \ud x  \\
      &\hspace{-1.3em} \stackrel{(\ref{eq:F}, \ref{eq:nablaF},
        \ref{eq:JF})}{\lesssim} \lambda \int_E \Bigl(R \sup
      \abs{\nabla \phi}
      + \phi \Bigr) \\
      & \lesssim \lambda \int_E \bigl( 1 + \phi\bigr), 
    \end{aligned}
  \end{equation*}
  where the last inequality uses $\sup \abs{\nabla \phi} \leq 1/R$.
  For the Coulomb repulsion term, we have
  \begin{equation}\label{eq:coulombest1}
    \frac{1}{2} \iint_{F_{\lambda}(E) \times F_{\lambda}(E)} \frac{1}{\abs{x - y}} \ud x \ud y = \frac{1}{2} \iint_{E \times E} \frac{1}{\abs{ F_{\lambda}(x) - F_{\lambda}(y)}} JF_{\lambda}(x) JF_{\lambda}(y) \ud x \ud y. 
  \end{equation}
  Note that by mean value theorem, 
  \begin{equation*}
    \abs{x - y} \leq \sup \abs{\nabla F_{\lambda}^{-1}}
    \abs{F_{\lambda}(x) - F_{\lambda}(y)} 
    \stackrel{\eqref{eq:nablaF}}{\leq} (1 + C \lambda) \abs{F_{\lambda}(x) - F_{\lambda}(y)}
  \end{equation*}
  for some universal constant $C$. Hence substituting this into
  \eqref{eq:coulombest1} and using \eqref{eq:JF}, we obtain
  \begin{equation*}
    \frac{1}{2} \iint_{F_{\lambda}(E) \times F_{\lambda}(E)} \frac{1}{\abs{x - y}} \ud x \ud y - \frac{1}{2} \iint_{E \times E} \frac{1}{\abs{x - y}} \ud x \ud y
    \lesssim \lambda \frac{1}{2} \iint_{E \times E}  \frac{1}{\abs{x - y}} \ud x \ud y. 
  \end{equation*}

  Finally, we consider the estimate for the perimeter, which is more
  subtle. Recall that (see e.g.,~\cite{Maggi:12}*{Proposition 17.1})
  if $E$ is a set of finite perimeter and $G: \RR^n \to \RR^n$ is a
  diffeomorphism, then $G(E)$ is still a set of finite perimeter, with
  (in classical analysis, this is the area formula)
  \begin{equation*}
    \nu_{G(E)} \mc{H}^{n-1} \mres \partial^{\ast} G(E) 
    = (G)_{\#} \bigl[ JG (\nabla G^{-1} \circ G)^{\ast} \nu_{E} \mc{H}^{n-1} \mres \partial^{\ast} E \bigr].
  \end{equation*}
  In particular, applying the above formula to $G = F_{\lambda}$, we
  get
  \begin{equation*}
    \begin{aligned}
      \bigl\lvert \partial F_{\lambda}(E) \cap B_R \bigr\rvert & =
      \int_{\partial^{\ast} E \cap B_R} \Biggl\lvert \Biggl( (1 +
      \lambda) \frac{f_{\lambda}(\abs{x})}{\abs{x}} \nu_E',
      \Bigl(\frac{f_{\lambda}(\abs{x})}{\abs{x}}\Bigr)^2 \nu_{E, r}
      \Biggr) \Biggr\rvert \ud \mc{H}^2, \\
      & \leq \int_{\partial^{\ast} E \cap B_R} \Bigl\lvert
      \Bigl( (1 + \lambda) \nu_E', \nu_{E, r}\Bigr) \Bigr\rvert \ud \mc{H}^2,
    \end{aligned}
  \end{equation*}
  where we write $\nu_E = (\nu_E', \nu_{E, r})$ with $\nu_{E, r}$
  parallel to the radial direction $\hat{x} = \frac{x}{\abs{x}}$ and
  $\nu_E' \perp \hat{x}$ and we have used in the inequality for $r
  \leq R$,
  \begin{equation*}
    \frac{f_{\lambda}(r)}{r} = 1 - \lambda \frac{R - r}{r} \leq 1. 
  \end{equation*}
  Using $\abs{\nu_E} = 1$ on $\partial^{\ast} E$, we have further 
  \begin{equation*}
    \begin{aligned}
      \bigl\lvert \partial F_{\lambda}(E) \cap B_R \bigr\rvert & \leq
      \int_{\partial^{\ast} E \cap B_R} \Bigl( (1 + \lambda)^2 ( 1 -
      \nu_{E, r}^2) + \nu_{E, r}^2 \Bigr)^{1/2} \ud \mc{H}^2 \\
      & \leq \int_{\partial^{\ast} E \cap B_R} \Bigl( 1 + 3 \lambda ( 1 -
      \nu_{E, r}^2) \Bigr)^{1/2} \ud \mc{H}^2 \\
      & \leq \int_{\partial^{\ast} E \cap B_R} \Bigl( 1 + 3
      \lambda ( 1 - \abs{\nu_{E, r}}) \Bigr)^{1/2} 
      \ud \mc{H}^2. 
    \end{aligned}
  \end{equation*}
  Subtracting $\abs{\partial E \cap B_R}$, we arrive at 
  \begin{equation*}
    \begin{aligned}
    \bigl\lvert \partial F_{\lambda}(E) \cap B_R \bigr\rvert -
      \bigl\lvert \partial E \cap B_R \bigr\rvert
& \leq \int_{\partial^{\ast} E \cap B_R} \Biggl[ \Bigl( 1 + 3
      \lambda ( 1 - \abs{\nu_{E, r}}) \Bigr)^{1/2} - 1\Biggr]
      \ud \mc{H}^2 \\
      & \lesssim \lambda \int_{\partial^{\ast} E \cap B_R} 1 -
      \abs{\nu_{E, r}} \ud \mc{H}^2 = \lambda \left(
        \bigl\lvert \partial E \cap B_R \bigr\rvert -
        \int_{\partial^{\ast} E \cap B_R} \abs{\nu_{E, r}} \ud
        \mc{H}^2 \right).
      \end{aligned}
  \end{equation*}
  In order to relate to $\abs{E \cap \partial B_R}$, recall the
  divergence theorem
  \begin{equation*}
    \int_{E} \divop T(x) \ud x = \int_{\partial^{\ast} E \cap B_R} T \cdot \nu_E \ud \mc{H}^2 - \int_{E \cap \partial B_R} T \cdot \hat{x} \ud \mc{H}^2. 
  \end{equation*}
  Applying this to $T = \hat{x} = \frac{x}{\abs{x}}$, using that
  $E\subset A_{\frac{3}{4} R, R}$ stays away from the singularity at
  $x = 0$, we get 
  \begin{equation*}
    \begin{aligned}
    \int_{\partial^{\ast} E \cap B_R} \abs{\nu_{E, r}} \ud
    \mc{H}^2 & \geq \left \lvert \int_{\partial^{\ast} E \cap B_R} T \cdot \nu_E \ud \mc{H}^2\right\rvert \\
    & \geq \abs{E \cap \partial B_R} -  \int_{E \cap B_R} \frac{2}{\abs{x}} \ud x \geq \abs{E \cap \partial B_R} - \frac{4}{R} \abs{E}.
    \end{aligned}
  \end{equation*}
  Substituting this into the previous inequality, we obtain the desired
  estimate for the perimeter. 
  This concludes the proof of the proposition.
\end{proof}

To apply Proposition~\ref{prop:deformboundary}, the following a priori
estimate will be useful. 
\begin{lemma}\label{lem:energyapriori}
  Under the same assumption as in Lemma~\ref{lem:moddensity}, we have 
  \begin{equation}\label{eq:energyapriori}
    \mc{E}_{\phi, R}(\Omega) - \abs{\Omega \cap \partial B_R} \lesssim 1.
  \end{equation}
\end{lemma}
\begin{proof}
  The idea is to compare $\Omega$ to the set $\Omega' = \Omega - B_R$,
  which however is not quite possible since it changes $\Omega$ near the
  boundary of $B_R$. Thus, we ``coat'' $\Omega \cap \partial B_R$ with a very
  thin layer and put the excessive volume at infinity, the optimality
  of $\Omega$ together with Lemma~\ref{lem:Eupperbound} then give us 
  \begin{equation*}
    \mc{E}_{\phi, R}(\Omega) - \abs{\Omega \cap \partial B_R} \leq  R^2 E_0(R^{-3} \abs{\Omega \cap B_R}) =  R^2 E_0(R^{-3}) \lesssim 1. 
  \end{equation*}
\end{proof}

We are now ready to prove Lemma~\ref{lem:moddensity}. The idea is to
compare the minimizer with a competitor by taking away the miminizer
in the annulus $A_{r-h/2, r+h/2}$ for
$r \in [\frac{3}{4} R, \frac{7}{8}R]$ and $h \in (0, 1]$. The volume
lost by cutting $V(h) = \abs{\Omega \cap A_{r-h/2, r+h/2}}$ is
compensated by deforming $\Omega - B_{r +h/2}$ as in
Proposition~\ref{prop:deformboundary}. The comparison then leads to a
differential inequality for $V(h)$, which gives a universal lower
bound of the volume $\abs{\Omega \cap A_{r-1/2, r+1/2}}$ for each
$r \in [\frac{3}{4} R, \frac{7}{8}R]$, and hence an upper bound of the
radius $R$.

\begin{proof}[Proof of Lemma~\ref{lem:moddensity} in the case
  $\abs{\Omega \cap B_{R_0}} \leq 1/2$]

  We will show that 
  \begin{equation}\label{eq:annulusbound}
    \abs{\Omega \cap A_{r - 1/2, r + 1/2}} \gtrsim 1, \qquad \text{for } r \in \bigl[\frac{3}{4}R, \frac{7}{8}R \bigr].
  \end{equation}
  This in turn gives an upper bound on $R$ as
  \begin{equation*}
    \frac{R}{8} \lesssim \int_{\frac{3}{4}R}^{\frac{7}{8}R} \abs{\Omega \cap A_{r - 1/2, r + 1/2}}
    \ud r \leq \abs{\Omega \cap B_R} = 1. 
  \end{equation*}
  Therefore, it suffices to show \eqref{eq:annulusbound}. 

  Fixing $r \in [\tfrac{3}{4} R, \tfrac{7}{8} R]$ and taking
  $h \in (0, 1]$, we denote
  \begin{equation*}
    V(h) = \bigl\lvert \Omega \cap A_{r - h/2, r + h/2} \bigr\rvert. 
  \end{equation*}
  We now compare the set $\Omega$ with $\Omega'$ given
  by (see Figure~\ref{fig:deform})
  \begin{equation*}
    \Omega' = F_{\lambda}(\Omega - B_{r + h/2})
  \end{equation*}
  for $\lambda$ chosen such that
  \begin{equation*}
    \abs{F_{\lambda}(\Omega - B_{r+h/2})} - \abs{\Omega - B_{r+h/2}} = 
    \abs{\Omega \cap A_{r-h/2, r+h/2}} = V(h),  
  \end{equation*}
  so that the volume of the competitor is the same as of $\Omega - B_{r - h/2}$.
  Since we may without loss of generality assume that 
  \begin{equation}\label{eq:smallVh}
    V(h) \ll 1,
  \end{equation}
  and since $\abs{\Omega - B_{r+h/2}} \geq \abs{\Omega \cap A_{R_0,
      R}} \geq 1/2$, this is possible by
  Proposition~\ref{prop:deformboundary} and we have
  \begin{equation}\label{eq:lambdabound}
    \lambda \lesssim V(h). 
  \end{equation}

  Using again Proposition~\ref{prop:deformboundary}, we thus obtain 
  \begin{align*}
    & \abs{ \partial \Omega' \cap B_R} - \abs{\partial (\Omega -
      B_{r+h/2}) \cap B_R} \lesssim \lambda \bigl( \abs{\partial
      (\Omega - B_{r+h/2}) \cap B_R} - \abs{\Omega \cap \partial B_R}
    + 4 R^{-1} \bigr) \\
    & \phantom{ \abs{\partial \Omega' \cap B_R} - \abs{\partial
        (\Omega - B_{r+h/2}) \cap B_R}} \lesssim \lambda \bigl(
    \abs{\partial \Omega \cap  B_R} - \abs{\Omega \cap \partial B_R} +
    \abs{ \Omega \cap \partial B_{r + h/2} } + 4 R^{-1}\bigr); \notag \\
    & \int_{\Omega'} \phi - \int_{\Omega - B_{r+h/2}} \phi \lesssim
    \lambda \int_{\Omega} ( \phi + 1); \\
    & \frac{1}{2} \iint_{\Omega' \times \Omega'} \frac{1}{\abs{x -
        y}}\ud x\ud y - \frac{1}{2} \iint_{(\Omega - B_{r+h/2}) \times
      (\Omega - B_{r+h/2})} \frac{1}{\abs{x - y}} \ud x\ud y\lesssim
    \lambda \frac{1}{2} \iint_{\Omega \times \Omega} \frac{1}{\abs{x -
        y}} \ud x\ud y.
  \end{align*}
  Adding up, we obtain
  \begin{equation}\label{eq:est1}
    \mc{E}_{\phi, R}\bigl(\Omega'\bigr) - \mc{E}_{\phi, R}\bigl( \Omega - B_{r
      + h/2} \bigr) \lesssim \lambda \bigl( 1 + \mc{E}_{\phi,
      R}(\Omega) - \abs{\Omega \cap \partial B_R} + \abs{\Omega
      \cap \partial B_{r+h/2}}\bigr) \lesssim \lambda \bigl( 1 +
    \abs{\Omega \cap \partial B_{r+h/2}}\bigr)
  \end{equation}
  where in the second inequality, we have used $\mc{E}_{\phi, R}(\Omega)
  - \abs{\Omega \cap \partial B_R} \lesssim 1$ from
  Lemma~\ref{lem:energyapriori}. By the optimality of $\Omega$, we have 
  \begin{equation}\label{eq:est2}
    \mc{E}_{\phi, R}(\Omega) \leq \mc{E}_{\phi, R}(\Omega') + R^{2} E_0\bigl(R^{-3} (\abs{\Omega} - \abs{\Omega'})\bigr) = 
    \mc{E}_{\phi, R}(\Omega') + R^{2} E_0\bigl(R^{-3} \abs{\Omega \cap B_{r - h/2}}\bigr). 
  \end{equation}
  Furthermore, by the definition of $\mc{E}_{\phi, R}$ and since $\phi
  > 0$ in $B_R$, we get 
  \begin{multline}\label{eq:est3}
    \mc{E}_{\phi, R}(\Omega) \geq  \mc{E}_{\phi, R}(\Omega - B_{r+h/2}) + 
    \mc{E}_{\phi, R}(\Omega \cap B_{r-h/2}) \\
    + \abs{\partial \Omega \cap A_{r - h/2, r + h/2}} - \abs{\Omega \cap \partial B_{r-h/2}}  - \abs{\Omega \cap \partial B_{r+h/2}}, 
  \end{multline}
  and also 
  \begin{equation}\label{eq:est4}
    R^{2} E_0\bigl(R^{-3} \abs{\Omega \cap B_{r-h/2}}\bigr) \leq \mc{E}_{\phi, R}(\Omega \cap B_{r - h/2}).
  \end{equation}
  Combining together the above four inequalities
  \eqref{eq:est1}--\eqref{eq:est4}, we  arrive at
  \begin{equation*}
    \abs{\partial \Omega \cap A_{r - h/2, r + h/2}} \leq \abs{\Omega \cap \partial B_{r-h/2}}  + \abs{\Omega \cap \partial B_{r+h/2}} 
    + C \lambda \bigl( 1 + \abs{\Omega
      \cap \partial B_{r+h/2}}\bigr).
  \end{equation*}
  Therefore, 
  \begin{equation*}
    \begin{aligned}
      \abs{\partial (\Omega \cap A_{r - h/2, r + h/2})} & \leq \abs{\partial \Omega \cap A_{r - h/2, r + h/2}} + \abs{\Omega \cap \partial B_{r-h/2}}  + \abs{\Omega \cap \partial B_{r+h/2}} \\
      & \lesssim \abs{\Omega \cap \partial B_{r-h/2}} +
      \abs{\Omega \cap \partial B_{r+h/2}} + \lambda.
    \end{aligned}
  \end{equation*}
  Using the isoperimetric inequality, we have
  \begin{equation*}
    \abs{\partial (\Omega \cap A_{r - h/2, r + h/2})} \gtrsim \abs{\Omega \cap A_{r - h/2, r + h/2}}^{2/3} =  V(h)^{2/3}. 
  \end{equation*}
  Hence, combined with \eqref{eq:lambdabound}, we get 
  \begin{equation*}
    V'(h) = \abs{\Omega \cap \partial B_{r-h/2}} +
    \abs{\Omega \cap \partial B_{r+h/2}} \geq \frac{1}{C_0} V(h)^{2/3} - C_0 V(h)
    \stackrel{\eqref{eq:smallVh}}{\gtrsim} V(h)^{2/3}.
  \end{equation*}
  Thus for $V(h) \lesssim 1$ the above inequality yields
  \begin{equation*}
    V'(h) \gtrsim  V(h)^{2/3}, \quad \text{or equivalently} \quad 
    \bigl(V(h)^{1/3}\bigr)' \gtrsim 1,
  \end{equation*}
  where we have used that $V(h) > 0$ for any $h > 0$, as otherwise,
  $\Omega \cap B_r$ (which is non-empty as $0 \in \Omega$) is
  disconnected with the rest of the set, which is impossible as the
  optimal $\Omega$ has to be connected inside $B_R$ by a similar
  argument as in the proof of Lemma~\ref{lem:connect}.  The last
  inequality implies the desired \eqref{eq:annulusbound} by
  integration. 
\end{proof}

\subsection{Proof of Lemma~\ref{lem:moddensity}: Case $\abs{\Omega  \cap B_{R_0}} \geq 1/2$}
\label{sec:interior}

After a suitable dilation by a factor of most two such that the volume
$\Omega \cap B_{R_0}$ expands to $1$, it suffices to prove
Lemma~\ref{lem:moddensity} under the additional assumption that
$\abs{\Omega \cap \partial B_R} \lesssim 1$.

Let $\Omega$ be the optimal set, using Lemma~\ref{lem:energyapriori}
and $\abs{\Omega \cap \partial B_R} \lesssim 1$, we have
\begin{equation}
  \mc{E}_{\phi, R}(\Omega) \leq  \bigl( \mc{E}_{\phi, R}(\Omega) - \abs{\Omega \cap \partial B_R} \bigr) + \abs{\Omega \cap \partial B_R} \lesssim 1,
\end{equation}
which, using the definition of $\mc{E}_{\phi, R}$ we upgrade to 
\begin{equation}\label{eq:wholeperim}
  \abs{\partial(\Omega \cap B_R)} \leq \abs{\partial \Omega
    \cap B_R} + \abs{\Omega \cap \partial B_R} \leq \mc{E}_{\phi, R}(\Omega) + \abs{\Omega \cap \partial B_R} \lesssim 1.
\end{equation}
Let us now show that we can find a ball of radius $1$ inside $B_R$
such that the optimal set has a non-trivial amount of volume inside
the ball, as stated in the following lemma.
\begin{lemma}\label{lem:ballexist}
  There exists a ball $B \subset B_R$ of radius $1$ such that
  \begin{equation*}
    \abs{ \Omega \cap B } \gtrsim 1, \quad 
    \abs{ B - \Omega} \gtrsim 1, \quad \text{and} \quad 
    \abs{ \partial \Omega \cap B} \lesssim 1.
  \end{equation*}
\end{lemma}
\begin{proof}
  
  The latter two requirements are satisfied by any choice of $B$,
  since $\abs{B - \Omega} \geq \abs{B} - \abs{\Omega \cap B_R} \geq
  \abs{B_1} - 1$ as $\abs{\Omega \cap B_R} = 1$, and $\abs{\partial
    \Omega \cap B} \leq \abs{\partial (\Omega \cap B_R)} \lesssim
  1$. Hence, it suffices to find a ball such that the first
  requirement $\abs{ \Omega \cap B } \gtrsim 1$ is satisfied.

  Consider a smooth symmetric convolution kernel $\varphi_{\delta}$ of
  support in $B_{\delta}$ with $\delta \ll 1$ to be determined. Note
  that for $\chi = \chi_{\Omega \cap B_R}$
  \begin{equation*}
    \int_{B_R} \abs{ \varphi_{\delta} \ast \chi - \chi} \lesssim \delta \int \abs{ \nabla \chi} \leq \delta \abs{ \partial(\Omega \cap B_R)} \lesssim \delta.
  \end{equation*}
  We choose $\delta$ sufficiently small such that 
  \begin{equation}\label{eq:convbound}
    \int_{B_R} \abs{ \varphi_{\delta} \ast \chi - \chi} < 1/2.
  \end{equation}
  This implies that at some point $x\in B_R$, we have
  $(\varphi_{\delta} \ast \chi)(x) > 1/2$. Since otherwise, if $\sup
  \varphi_{\delta} \ast \chi \leq 1/2$, we would have
  \begin{equation*}
    \int_{B_R} \abs{ \varphi_{\delta} \ast \chi - \chi } \geq \int_{\Omega} \frac{1}{2} = \frac{1}{2},
  \end{equation*}
  which contradicts \eqref{eq:convbound}. It follows then $B_1(x)$ is
  a ball satisfying all the conditions.
\end{proof}

We now deform the set inside $B$ in a way that the perimeter increases
not more than proportionally to the increase of the volume. This is a
crucial proposition which plays the same role as
Proposition~\ref{prop:deformboundary} in Section~\ref{sec:boundary}.
Actually, the remainder of the proof of Lemma~\ref{lem:moddensity} is
similar to that in Section~\ref{sec:boundary} after
Proposition~\ref{prop:deformboundary}: Instead of deforming the set
$\Omega - B_{r + h/2}$, we deform the set $\Omega \cap B$ given in
Lemma~\ref{lem:ballexist} to compensate for the volume in
$\abs{\Omega \cap A_{r-h/2, r+h/2}}$. We will omit the details, as the
arguments are parallel.
\begin{prop}\label{prop:vector}
  Suppose we are given a set $\Omega$ and a radius-$1$
  ball $B$ with 
  \begin{equation*}
    \abs{ \Omega \cap B } \gtrsim 1, \quad 
    \abs{ B - \Omega} \gtrsim 1, \quad \text{and} \quad 
    \abs{ \partial \Omega \cap B} \lesssim 1.
  \end{equation*}
  Then for every $\lambda \ll 1$, there exists a set $\Omega'$ with
  $\Omega' \Delta \Omega \subset\subset B$ and
  \begin{align}
    \label{eq:vector1} \abs{\Omega'} - \abs{\Omega} & \gtrsim \lambda; \\
    \label{eq:vector2} \abs{\partial \Omega' \cap B} - 
    \abs{\partial \Omega \cap B} & \lesssim \lambda; \\
    \label{eq:vector3} \int_{\Omega'} \phi - \int_\Omega \phi & \lesssim \lambda  \left( 1 + \int_{\Omega} \phi \right); \\
    \label{eq:vector4} \frac{1}{2} \iint_{\Omega' \times \Omega'}
    \frac{1}{\abs{x - y}} \ud x \ud y- \frac{1}{2} \iint_{\Omega \times \Omega}
    \frac{1}{\abs{x - y}} \ud x \ud y & \lesssim \lambda \frac{1}{2} \iint_{\Omega \times
      \Omega} \frac{1}{\abs{x - y}} \ud x \ud y.
  \end{align}
\end{prop}

The proof of Proposition~\ref{prop:vector} follows a similar argument
as the proof of Proposition~\ref{prop:deformboundary}. Instead of
using an explicit formula to define $F_{\lambda}$, the deformation
$F_{\lambda}(x)$ in the current case is given by the solution map
generated by a vector field $\xi$:
\begin{equation*}
  \frac{\ud}{\ud \lambda} F_{\lambda} = \xi\bigl(F_{\lambda}\bigr), \qquad F_0(x) = x, 
\end{equation*}
where the vector field is constructed by the following lemma. 
\begin{lemma}\label{lem:vector}
  Under the assumptions of Proposition~\ref{prop:vector}, there exists
  $\xi: B \to \RR^3$ compactly supported such that
  \begin{equation*}
    \int_{\Omega} \nabla \cdot \xi = 1 
    \qquad \text{and} \qquad
    \norm{\xi}_{C^2(B)} \lesssim 1.
  \end{equation*}
\end{lemma}

Let us first prove the proposition assuming Lemma~\ref{lem:vector}.
\begin{proof}[Proof of Proposition~\ref{prop:vector}]
  We take $\Omega' = F_{\lambda}(\Omega)$ for the deformation
  $F_{\lambda}$ constructed above. Using $\norm{\xi}_{C^2(B)} \lesssim
  1$ from Lemma~\ref{lem:vector}, we have for $x \in \Omega$
  \begin{equation}\label{eq:estF2}
    \abs{F_{\lambda}(x) - x} \lesssim \lambda,  \quad 
    \abs{\nabla F_{\lambda}(x) - \mathrm{Id}} \lesssim \lambda, 
    \quad \text{and} \quad 
    \abs{J F_{\lambda}(x) - 1} \lesssim \lambda. 
  \end{equation}
  Thus, the estimates \eqref{eq:vector3} and \eqref{eq:vector4} follow
  from similar calculations in the proof of
  Proposition~\ref{prop:deformboundary}. The perimeter estimate is in
  fact more straightforward now as $\abs{\partial(\Omega \cap B)}
  \lesssim 1$, and hence we will omit the details.

  For the volume, note that by Liouville's formula and
  Lemma~\ref{lem:vector}, we have
  \begin{equation*}
    \frac{\rd}{\rd \lambda} \abs{F_{\lambda}(\Omega)} = \frac{\rd}{\rd \lambda} \int_{\Omega}  J F_{\lambda}(x) \ud x  = \int_{\Omega} (\nabla \cdot \xi)(F_{\lambda}(x)) J F_{\lambda}(x) \ud x. 
  \end{equation*}
  Since $\int_{\Omega} \nabla \cdot \xi = 1$ from
  Lemma~\ref{lem:vector}, we have
  \begin{equation*}
    \begin{aligned}
      \Biggl\lvert \frac{\rd}{\rd \lambda} \abs{F_{\lambda}(\Omega)} -
      1 \Biggr\rvert & = \Biggl\lvert \int_{\Omega} (\nabla \cdot
      \xi)(F_{\lambda}(x)) J F_{\lambda}(x) \ud x
      - \int_{\Omega} (\nabla \cdot \xi)(x) \ud x \Biggr\rvert \\
      & \leq \Biggl\lvert \int_{\Omega} (\nabla \cdot \xi)(x)
      (JF_{\lambda}(x) - 1) \ud x \Biggr\rvert + \Biggl\lvert
      \int_{\Omega} \Bigl( (\nabla \cdot \xi)(F_{\lambda}(x)) -
      (\nabla \cdot \xi)(x)\Bigr)J F_{\lambda}(x) \ud x \Biggr\rvert \\
      & \leq \norm{\xi}_{C^1(B)} \int_{\Omega} \abs{J F_{\lambda} - 1}
      + \norm{\xi}_{C^2(B)} \int_{\Omega}
      \abs{F_{\lambda}(x) - x} \abs{ JF_{\lambda}(x)} \ud x \\
      & \stackrel{\eqref{eq:estF2}}{\lesssim} \lambda
      \norm{\xi}_{C^2(B)} \lesssim \lambda,
    \end{aligned}
  \end{equation*}
  where the last inequality follows from
  $\norm{\xi}_{C^2(B)} \lesssim 1$.  Hence, we arrive at the estimate
  \eqref{eq:vector1} as
  \begin{equation*}
    \abs{\Omega'} - \abs{\Omega} = \abs{F_{\lambda}(\Omega)} - \abs{\Omega} = \int_0^{\lambda} \frac{\rd}{\rd \lambda'} \abs{F_{\lambda'}(\Omega)} \ud \lambda' \gtrsim \lambda. 
  \end{equation*}
\end{proof}

We conclude this section by proving Lemma~\ref{lem:vector}.
\begin{proof}[Proof of Lemma~\ref{lem:vector}]
  For $\delta \ll 1$ to be fixed later consider
  \begin{itemize}
  \item a smooth radial cut-off function $\eta_{\delta}$ of
    $B_{1-\delta}$ in $B$;
  \item a smooth symmetric convolution kernel $\varphi_{\delta}$ of
    support in $B_{\delta}$.
  \end{itemize}
  Let $\wb{f}$ denote the spatial average of $f$ on $B$.  Let $\chi$
  denote the characteristic function of $\Omega$, for simplicity.

  Solve the Neumann problem
  \begin{align}
    \label{eq:v} & - \Delta v = \varphi_{\delta} \ast \left[
      \eta_{\delta}^2 \left( \chi - \frac{\wb{\eta_{\delta}^2
            \chi}}{\wb{\eta_{\delta}^2}}\right)\right] &&
    \text{in }B, \\
    \label{eq:vbc} & \nu \cdot \nabla v = 0 && \text{on }\partial B.
  \end{align}
  Note that it is solvable because the r.h.s.{} has average zero on
  $\RR^d$ and is compactly supported in $B$. We set
  \begin{equation*}
    \xi = - \eta_{\delta} \nabla v
  \end{equation*}
  and note that $\xi$ is smooth uniformly in $\chi$ for fixed $\delta
  > 0$. More precisely, since $ \eta_{\delta}^2 \left( \chi -
    \frac{\wb{\eta_{\delta}^2 \chi}}{\wb{\eta_{\delta}^2}}\right)$ is
  bounded in $L^2$ uniformly in $\chi$, the r.h.s.{} $\varphi_{\delta} \ast \left[ \eta_{\delta}^2 \left( \chi -
      \frac{\wb{\eta_{\delta}^2
          \chi}}{\wb{\eta_{\delta}^2}}\right)\right]$ of \eqref{eq:v}
  is bounded in any norm uniformly in $\chi$. Therefore, by elliptic
  regularity theory $v$ and thus $\xi$ is bounded in any norm, in
  particular $\norm{\cdot}_{C^2}$ uniformly in $\chi$ for fixed
  $\delta > 0$.

  Hence, it is enough to show 
  \begin{equation*}
    \int_{\Omega} \nabla \cdot \xi \geq \frac{1}{C} \min \bigl\{ \abs{B \cap \Omega}, \abs{B - \Omega} \bigr\} - C \delta \abs{B \cap \partial \Omega} 
    - C \delta^{1/2}. 
  \end{equation*}
  Indeed,
  \begin{equation*}
    \begin{aligned}
      \int_{\Omega} \nabla \cdot \xi & = \int_B \chi \nabla \cdot \xi
      = - \int_B \eta_{\delta} \chi \Delta v - \int_B \chi \nabla \eta_{\delta} \cdot \nabla v \\
      & \geq \int_B \eta_{\delta}^2 \left(\chi -
        \frac{\wb{\eta_{\delta}^2 \chi}}{\wb{\eta_{\delta}^2}}\right)
      \varphi_{\delta} \ast (\eta_{\delta} \chi) - \int_B \abs{\nabla
        \eta_{\delta} \cdot \nabla v} \\
      & \geq \int_B \eta_{\delta}^2 \left(\chi -
        \frac{\wb{\eta_{\delta}^2 \chi}}{\wb{\eta_{\delta}^2}}\right)
      \eta_{\delta} \chi - \int_B \abs{\varphi_{\delta} \ast
        (\eta_{\delta} \chi) - \eta_{\delta} \chi} - \int_B
      \abs{\nabla \eta_{\delta} \cdot \nabla v}.
    \end{aligned}
  \end{equation*}
  We consider the first term and note 
  \begin{multline*}
    \int_B \eta_{\delta}^2 \left(\chi - \frac{\wb{\eta_{\delta}^2
          \chi}}{\wb{\eta_{\delta}^2}}\right) \eta_{\delta} \chi =
    \frac{\abs{B}}{\wb{\eta_{\delta}^2}} \wb{\eta_{\delta}^3 \chi}
    \overline{\eta_{\delta}^2 (1 - \chi)} \gtrsim \abs{B_{1-\delta}
      \cap \Omega} \abs{B_{1-\delta} - \Omega} \\
    \geq \frac{1}{C} (\abs{B \cap \Omega} - C \delta) (\abs{B - \Omega} - C
    \delta) \geq \frac{1}{C} \min \bigl\{ \abs{B \cap \Omega}, \abs{B - \Omega}
    \bigr\} - C \delta.
  \end{multline*}
  We now treat the second term:
  \begin{equation*}
    \int_B \abs{\varphi_{\delta} \ast (\eta_{\delta} \chi) - \eta_{\delta} \chi}
    \lesssim \delta \int \abs{\nabla (\eta_{\delta} \chi)}
    \lesssim \delta ( \abs{B \cap \partial \Omega} + 1 ).
  \end{equation*}
  We finally address the last term: We note that the r.h.s.{} of
  \eqref{eq:v} is bounded in $L^2$ uniformly in $\delta$ and
  $\chi$. By $H^2$-regularity, this implies
  \begin{equation*}
    \int_B \abs{\nabla^2 v}^2 \lesssim 1. 
  \end{equation*}
  In view of the boundary condition \eqref{eq:vbc}, this implies by
  Hardy's inequality
  \begin{equation*}
    \int_B \frac{1}{(1 - \abs{x})^2} \abs{x\cdot \nabla v}^2 \ud x \lesssim 1. 
  \end{equation*}
  Since $\eta_{\delta}$ is radially symmetric: 
  \begin{equation*}
    \begin{aligned}
      \int_B \abs{\nabla \eta_{\delta} \cdot \nabla v} & \lesssim
      \frac{1}{\delta} \int_{B - B_{1-\delta}} \abs{x \cdot \nabla v} \ud x \\
      & \lesssim \frac{1}{\delta} \left( \int_{B - B_{1 - \delta}} (1
        - \abs{x})^2 \ud x \int_{B - B_{1 - \delta}} \frac{1}{(1 -
          \abs{x})^2}  \abs{x \cdot \nabla v}^2 \ud x \right)^{1/2} \\
      & \lesssim \frac{1}{\delta} (\delta^3)^{1/2} = \delta^{1/2}.
    \end{aligned}
  \end{equation*}
  The proof is concluded by combining the above estimates. 
\end{proof}

\section{Proof of Theorem~\ref{thm:ball}}
\label{sec:ball}

We now prove Theorem~\ref{thm:ball}.  It is more convenient to rescale
the problem such that we consider sets with volume equal to a unit
ball.
\begin{equation}
  E_{V, Z}(\Omega) :=  \abs{\partial \Omega} + \frac{V}{\abs{B_1}} \frac{1}{2} \iint_{\Omega \times \Omega} \frac{1}{\abs{x-y}} \ud x \ud y - Z \int_{\Omega} \frac{1}{\abs{x}} \ud x.
\end{equation}

The key element of the proof is the version of the quantitative
isoperimetric inequality \cites{Julin:14}:
\begin{lemma}\label{lem:quantisop}
  There exists a universal constant $c_{\text{isop}}$ such that
  $\forall\;\Omega\subset \RR^3$ with volume $\abs{\Omega} = \abs{B_1}$,
  \begin{equation}
    \abs{\partial \Omega} -\abs{\partial B_1} \geq c_{\text{isop}} \gamma(\Omega)
  \end{equation}
  where $\gamma(\Omega)$ is defined as 
  \begin{equation*}
    \gamma(\Omega) :=  
    \int_{B_1} \frac{1}{\abs{x}} \ud x - \int_{\Omega} \frac{1}{\abs{x}} \ud x.
  \end{equation*}
\end{lemma}

\begin{proof}[Proof of Theorem~\ref{thm:ball}]
Define 
\begin{equation*}
  u = \frac{1}{\abs{x}} \ast (\chi_{B_1} - \chi_{\Omega}).
\end{equation*}
It is clear that $u$ is a superharmonic function in $B_1$. 
Note that we have 
\begin{equation}\label{eq:attracdiff}
  u(0) = \int_{B_1} \frac{1}{\abs{x}} \ud x - \int_{\Omega} \frac{1}{\abs{x}} \ud x = \gamma(\Omega), 
\end{equation}
and also
\begin{equation}\label{eq:coulombdiff}
  \begin{aligned}
    \iint_{B_1 \times B_1} \frac{1}{\abs{x -y}} \ud x\ud y -
    \iint_{\Omega \times \Omega} \frac{1}{\abs{x -y}} \ud x\ud y &
    \leq \iint_{\RR^3 \times \RR^3}
    \frac{\chi_{B_1}(x)}{\abs{x-y}} {\bigl(\chi_{B_1}(y) - \chi_{\Omega}(y)\bigr)} \ud x\ud y \\
    & = \int_{B_1} u \leq \abs{B_1} u(0),
  \end{aligned}
\end{equation}
where the last inequality follows as $u$ is superharmonic in
$B_1$. 

Now if $V-Z < c_{\text{isop}}$, we have for any $\Omega \subset \RR^3$
that $\abs{\Omega} = \abs{B_1}$, 
\begin{equation*}
  \begin{aligned}
    E_{V, Z}(B_1) - E_{V, Z}(\Omega) & \leq \abs{\partial B_1} -
    \abs{\partial \Omega} + (V - Z) u(0) \\
    & \leq - c_{\text{isop}} \gamma(\Omega) + (V - Z) \gamma(\Omega)
    \leq 0,
  \end{aligned}
\end{equation*}
where the first inequality follows from \eqref{eq:attracdiff} and
\eqref{eq:coulombdiff} and the second follows from
Lemma~\ref{lem:quantisop}. This concludes the proof that the optimal
set $\Omega$ must be the ball $B_1$.
\end{proof}

\bibliographystyle{amsxport}
\bibliography{TFDW}

\end{document}